%% file: main.tex
\PassOptionsToPackage{dvipsnames}{xcolor}
\documentclass[leqno]{siamart1116}
\usepackage{./Macros/mydef}
\usepackage{amssymb}
\usepackage{graphicx}
\usepackage{xspace}
\usepackage{derivative}
\usepackage{bm,bbm}
\usepackage[numbers]{natbib}
\usepackage{dsfont}
\usepackage{enumitem}
\usepackage{subcaption}
\usepackage{booktabs}
\usepackage[algo2e,linesnumbered]{algorithm2e}
\usepackage{algorithmic}

\usepackage{siunitx}
\newcommand{\mynum}[2]{{\qty[scientific-notation=false, round-mode=figures,round-precision = 5, drop-zero-decimal, round-pad = false]{#1}{#2}}}
\newcommand{\myhide}[1]{{}}

\usepackage{graphicx}
\usepackage{graphics}
\usepackage{wrapfig}
\usepackage{float}
\usepackage[percent]{overpic}
\usepackage{varwidth}
\usepackage{tikz}
\usepackage{pgfplots}
\usepackage{adjustbox}
\usetikzlibrary{arrows,matrix,positioning,fit}
\usetikzlibrary{shapes,positioning}
\usetikzlibrary{backgrounds}
\usepackage{tikz-layers}
\usepgfplotslibrary{fillbetween}
\usetikzlibrary{intersections}
\pgfplotsset{compat=1.14}
\usepgfplotslibrary{colorbrewer}
\usepgfplotslibrary{patchplots}
\usepgfplotslibrary[colorbrewer]
\usetikzlibrary{pgfplots.colorbrewer}
\usetikzlibrary[pgfplots.colorbrewer]
\usepgfplotslibrary{units}
\usetikzlibrary{spy}
\usepackage{pgfplotstable}
\usepackage{arrayjobx}
\graphicspath{ {./figs/} }
\usetikzlibrary{external}

\colorlet{PlotColor1}{Spectral-A}
\colorlet{PlotColor3}{Spectral-L}
\colorlet{PlotColor2}{Spectral-O}
\colorlet{PlotColor4}{Spectral-D}
\colorlet{PlotColor5}{black}

\begin{document}

\newcommand\footnotemarkfromtitle[1]{%
	\renewcommand{\thefootnote}{\fnsymbol{footnote}}%
	\footnotemark[#1]%
	\renewcommand{\thefootnote}{\arabic{footnote}}}

\newcommand{\TheTitle}{Second-order invariant-domain preserving approximation to the multi-species Euler equations}
\newcommand{\TheAuthors}{B. Clayton, T. Dzanic, E.~J. Tovar}

\headers{Multi-species invariant-domain preserving approximation}{\TheAuthors}

\title{{\TheTitle}\thanks{Draft version, \today \funding{
			ET acknowledges the support from the U.S. Department of Energy’s Office of Applied Scientific Computing Research (ASCR) and Center for Nonlinear Studies (CNLS) at Los Alamos National Laboratory (LANL) under the Mark Kac Postdoctoral Fellowship in Applied Mathematics. ET further acknowledges the Competitive Portfolios program through ASCR. BC acknowledges the support of the Eulerian Applications Project (within the Advanced Simulation and Computing program) at LANL.
			LANL is operated by Triad National Security, LLC, for the National Nuclear Security Administration of the U.S. Department of Energy (Contract No. 89233218CNA000001). LANL release number LA-UR-25-24617.  This work was partially performed under the auspices of the U.S. Department of Energy by Lawrence Livermore National Laboratory under contract DE--AC52--07NA27344. Release number LLNL--JRNL--2005960.}}}

\author{
	Bennett Clayton\footnotemark[3]
	\and
	Tarik Dzanic\footnotemark[1]
	\and
	Eric J. Tovar\footnotemark[2]
}

\maketitle

\renewcommand{\thefootnote}{\fnsymbol{footnote}}
\footnotetext[1]{Center for Applied Scientific Computing, Lawrence Livermore National Laboratory, Livermore, CA 94551, USA}
\footnotetext[2]{Theoretical Division, Los Alamos National Laboratory, P.O. Box 1663, Los Alamos, NM, 87545, USA.}
\footnotetext[3]{X Computational Division, Los Alamos National Laboratory, P.O. Box 1663, Los Alamos, NM, 87545, USA.}

\renewcommand{\thefootnote}{\arabic{footnote}}
\begin{abstract}
	This work is concerned with constructing a second-order, invariant-domain preserving approximation of the compressible multi-species Euler equations where each species is modeled by an ideal gas equation of state. We give the full solution to the Riemann problem and derive its maximum wave speed. The maximum wave speed is used in constructing a first-order invariant-domain preserving approximation. We then extend the methodology to second-order accuracy and detail a convex limiting technique which is used for preserving the invariant domain. Finally, the numerical method is verified with analytical solutions and then validated with several benchmarks and laboratory experiments.
\end{abstract}



\input{introduction.tex}
\input{model_problem.tex}
\input{riemann_problem.tex}
\input{solution_RP.tex}
\input{spatial.tex}

\input{limiting.tex}
\input{illustrations.tex}
\input{conclusion.tex}

\appendix
\input{app_PTE.tex}

\bibliographystyle{abbrvnat}
\bibliography{ref}

\end{document}

%% file: introduction.tex
\section{Introduction}\label{sec:introduction}
The understanding of fluid interactions between multiple miscible fluids remains a vital component of many engineering applications. For example, accurate modeling of multi-species fluid mixing is crucial in combustion research for predicting flame stability, ignition, and emissions characteristics, which directly affect engine efficiency and pollutant formation. Furthermore, in climate modeling, understanding how different atmospheric gases and moisture interact and mix is essential for predicting weather patterns and climate change impacts. Such fluid mixing phenomena are also prevalent in fields like aerospace engineering, nuclear engineering, and materials science, where accurate modeling directly impacts system performance and safety.

Accurately simulating these complex flow phenomena has proven challenging for many numerical methods, particularly for applications involving high-speed flows where compressibility effects must be accounted for. Multi-species flows at high Mach numbers introduce additional challenges due to phenomena such as shock waves, contact discontinuities, rarefaction waves, and species interfaces. Developing robust and high-fidelity numerical techniques capable of effectively capturing this behavior remains an active area of research. A variety of numerical approaches have been proposed for modeling multiphase compressible flows, with methods typically falling into the categories of: \textup{i)} mixture fraction models, where mixing is modeled by the evolution of a scalar conserved field (\ie the ``mass/mixture fraction'') without resolving individual species equations~\citep{Mulder1992}; \textup{ii)} multi-species (or multi-component) models, where the evolution of individual species densities is modeled in addition to the momentum and energy of the mixture~\citep{larrouturou2006}; and \textup{iii)} multi-fluid models, where each species is modeled as a separate fluid with individual conservation equations and interspecies interaction terms~\citep{Truesdell1984}. 

The focus of this work is on the inviscid limit of the second approach -- the multi-species Euler equations. In particular, we are concerned with conservative approximations of the compressible multi-species Euler equations, where the equation of state of each species is modeled as an ideal gas. In the context of numerical schemes for hyperbolic conservation laws, namely gas dynamics, there has been growing interest in the development of numerical schemes that are \emph{provably robust} in the sense that they guarantee the solution abides by known physical laws. For the multi-species Euler equations, this corresponds to the preservation of physical invariants such as positivity of species densities, positivity of internal energy (and consequently pressure), and adherence to an entropy inequality. Early iterations of numerical schemes which preserve positivity of species densities were shown in works such as that of~\citet{Larrouturou1991}. To the authors' knowledge, the first development of schemes which \emph{provably} preserve positivity of internal energy/pressure in multi-species flows was shown in~\citet{Shahbazi2017}. More recently, the minimum entropy principle for entropy solutions of the multi-species Euler equations was proven in \citet{gouasmi2020minimum}, which has led to the development of entropy-stable numerical schemes that adhere to an entropy inequality for such equations~\citep{Gouasmi2020, Peyvan2023, Trojak2024, Ching2024a, Ching2024b}. 

The primary novelty of this work is the construction a second-order, invariant-domain preserving approximation of the compressible multi-species Euler equations which preserves physical invariants such as the positivity of the individual species densities and internal energy/pressure as well as a local minimum principle on the mixture entropy. 
We first give the full solution to the Riemann problem and derive the local maximum wave speed -- which, to the authors' knowledge, has never been explicitly derived before.  
We  show that it is possible to estimate an upper bound on the maximum wave speed of the one-dimensional Riemann problem in the multi-species model, and then use this to construct a first-order invariant-domain preserving approximation in the manner of~\citet{guermond2016invariant}. We then describe a modified convex limiting technique~\citep{guermond2018second} to blend a nominally second-order accurate scheme with the first-order invariant-domain preserving approximation in such a way to retain both the invariant-domain preserving properties and second-order accuracy. The efficacy of the proposed scheme is shown in various numerical experiments involving multi-species flows with strong shocks. 

The remainder of this manuscript is organized as follows. We first describe the multi-species model and thermodynamic assumptions in~\Cref{sec:model_problem}. We then give the full solution to the Riemann problem and derive the maximum wave speed for the Riemann problem in~\Cref{sec:riemann_problem}. The construction of a low-order invariant-domain preserving approximation and a provisional high-order approximation is introduced in~\Cref{sec:approximation}, and the convex limiting technique is then described in~\Cref{sec:limiting}. The main result is given in Theorem~\ref{thm:idp_property}. Finally, we present the results of numerical experiments in~\Cref{sec:illustrations}.

%% file: model_problem.tex
\section{The model problem}\label{sec:model_problem}
Let $D\subset\polR^d$ be a polygonal domain where $d=\{1,2,3\}$ is the spatial dimension. We consider a mixture of $n_s\geq 2$ compressible, inviscid species occupying $D$. 
We assume that at all times there exists at least one species in a subset of $D$. 
We further assume that all species are in thermal and mechanical equilibrium and ignore any effects due to molecular diffusion. We define the conserved variable of the mixture system in question by: $\bu(\bx,t)\eqq((\alpha_1\rho_1,\dots,\alpha_{n_s}\rho_{n_s}),\bbm,E)\tr(\bx,t)$. 
Here, $\alpha_k \rho_k$ is the conserved partial density for each species where $k\in\{1,\dots, n_s\}$, $\bbm$ is the mixture momentum and $E$ is the total mechanical energy of the mixture.
We further define the following mixture quantities: density -- $\rho(\bu)\eqq\spsum\alpha_k\rho_k$; velocity -- $\bv(\bu)\eqq\bbm/\rho$; internal energy --  $\varepsilon(\bu)\eqq E - \frac{1}{2\rho}\|\bbm\|_{\ell^2}^2$; and specific internal energy -- $e(\bu)\eqq \varepsilon(\bu)/\rho(\bu)$.
The goal of this work is as follows. 
Given some initial data $\bu_0(\bx) \eqq ((\alpha_1\rho_1,\dots,\alpha_{n_s}\rho_{n_s})_0, \bbm_0, E_0)(\bx)$ at time $t_0$, we seek solutions that solve the following system in some weak sense:
\begin{subequations}\label{eq:model}
\begin{align}
    &\partial_t (\alpha_k \rho_k) + \DIV(\bv \alpha_k \rho_k) = 0, \qquad k\in\{1,\dots, n_s\},\\
    &\partial_t \bbm + \DIV(\bv \otimes \bbm + p(\bu)\polI_d) = \bm{0}, \\
    &\partial_t E + \DIV(\bv(E + p(\bu))) = 0.
\end{align}
\end{subequations}
Here, $\polI_d$ is the $d\times d$ identity matrix.
We introduce the short-hand notation for the flux of the system: $\polf(\bu)\eqq (\bv \alpha_k\rho_k, \bv\otimes\bbm +  p(\bu)\polI_d, \bv(E + p(\bu)))\tr$.
We also define the mass fractions of each species by $Y_k \eqq \alpha_k \rho_k / \rho(\bu)$ and the respective volume fractions by $\alpha_k \eqq \alpha_k \rho_k / \rho_k$ where $\rho_k$ is the material density defined through the equation of state (see Remark~\ref{rem:mat_quantities}).
We denote the mass fraction and volume fraction vector quantities by $\bY(\bu) \eqq (Y_1, \ldots, Y_{n_s})^{\sfT}$ and $\balpha(\bu) \eqq (\alpha_1, \ldots, \alpha_{n_s})^{\sfT}$, respectively.
For the sake of simplicity, we may drop the dependence on $\bu$ when discussing the mass and volume fractions.
\begin{remark}[Alternative formulation]
    It is common in the literature to formulate the system~\eqref{eq:model} in terms of the species mass fractions.
    Either in terms of the $n_s - 1$ species mass fractions with the bulk density $\rho$ (see:~\cite[Eqn.~1.1]{larrouturou2006}):
    \begin{subequations} \label{eq:mass-fraction-model1}
    \begin{align}
        &\partial_t (\rho Y_k) + \DIV(\bv \rho Y_k) = 0, \qquad k\in\{1,\dots, n_s-1\},\\
        &\partial_t \rho + \DIV(\bv \rho) = 0, \\
        &\partial_t \bbm + \DIV(\bv \otimes \bbm + p(\bu)\polI_d) = \bm{0}, \\
        &\partial_t E + \DIV(\bv(E + p(\bu))) = 0,
    \end{align}
    \end{subequations}
    or just in terms of each species mass fraction:
    \begin{subequations} \label{eq:mass-fraction-model2}
    \begin{align}
        &\partial_t (\rho Y_k) + \DIV(\bv \rho Y_k) = 0, \qquad k\in\{1,\dots, n_s\},\\
        &\partial_t \bbm + \DIV(\bv \otimes \bbm + p(\bu)\polI_d) = \bm{0}, \\
        &\partial_t E + \DIV(\bv(E + p(\bu))) = 0.
    \end{align}
    \end{subequations}
    As stated in~\citet[Remark~1]{larrouturou2006}, each of these system formulations, \eqref{eq:model}, \eqref{eq:mass-fraction-model1}, and \eqref{eq:mass-fraction-model2} are equivalent.
\end{remark}

\subsection{Thermodynamics}
We assume that each species is governed by an ideal gas equation of state. We further assume that the system is in thermal equilibrium and mechanical equilibrium. Then, the bulk pressure for the ideal mixture is given by:
\begin{equation}\label{eq:pressure}
    p(\bu) \eqq (\gamma(\bY) - 1)\varepsilon(\bu), \text{ where } \, \gamma(\bY) \eqq \frac{\spsum Y_k c_{p,k}}{\spsum Y_k c_{v,k}} = \frac{c_p(\bY)}{c_v(\bY)},
\end{equation}
denotes the mixture adiabatic index and $\{c_{p,k}\}_{k=1}^{n_s}$ and $\{c_{v,k}\}_{k=1}^{n_s}$ are the material specific heat capacities at constant pressure and volume, respectively.
Note that $\gamma$, $c_v$, and $c_p$ are technically functions of the conservative variables $\bu$; however, we write only their dependence on $\bY$ for simplicity and clarity.
We define the ratio of specific heats by: $\gamma_k\eqq c_{p,k}/c_{v,k}$ and the specific gas constant by: $r_k \eqq c_{p,k} - c_{v,k}$ for all $k \in \intset{1}{n_s}$.
Note that in ~\cite[Prop.~2]{larrouturou2006} it was shown that the system~\eqref{eq:model} is hyperbolic if $\gamma_k > 1$ holds for each species.

Since each species is governed by an ideal gas, the temperature is defined by $c_{v,k} T = e_k$ and taking the mass fraction average we see that $\spsum Y_k e_k = c_v(\bY) T$.
Similarly we can take the mass fraction average of the $p$-$T$-$\rho_k$ relation, $\frac{p}{\rho_k} = r_k T$, to find $\frac{p}{\rho} = r(\bY)T$ where we used the identity, $\spsum \frac{Y_k}{\rho_k} = \spsum \frac{\alpha_k}{\rho} = \frac{1}{\rho}$ and $r(\bY) \eqq c_p(\bY) - c_v(\bY)$.
The specific internal energy identity follows:
\[
    \spsum Y_k e_k =  c_v(\bY) T = \frac{p c_v(\bY)}{r(\bY) \rho} = \frac{p}{(\gamma(\bY) - 1)\rho} = e(\bu).
\]
In general, the temperature is recovered from:
\begin{equation} \label{eq:temperature}
    T(\bu) \eqq \frac{e(\bu)}{c_v(\bY)}.
\end{equation}
The mixture entropy is defined by $\rho s(\bu) = \spsum \alpha_k \rho_k s_k(\rho_k, e_k)$.
It is assumed that each material satisfies the Gibbs identity: $T \diff s_k = \diff e_k + p \diff \tau_k$, where $\tau_k = \rho_k^{-1}$. Hence, the specific entropy for each material, $k \in \intset{1}{n_s}$, is given by, $s_k = c_{v,k} \log(e_k / \rho_k^{\gamma_k-1}) + s_{\infty, k}$, where $s_{\infty,k}$ is some reference specific entropy. 
The mixture specific entropy is defined by:
\begin{equation} \label{eq:mix_entropy}
    s \eqq \spsum Y_k s_k = \spsum Y_k c_{v,k}\log\Big(\frac{e_k}{\rho_k^{\gamma_k - 1}}\Big) + Y_k s_{\infty,k}.
\end{equation}
Using each material Gibbs' identity and and the mass fraction average of the material specific entropies, we have the mixture Gibbs' identity:
\begin{equation}\label{eq:gibbs_mixture}
    T \diff s = \diff e + p \diff \tau - \spsum (e_k + p\tau_k - T s_k) \diff Y_k.
\end{equation}
This is used in the derivation of the solution to the Riemann problem in Section~\ref{sec:riemann_problem}.

Notice that the mixture entropy~\eqref{eq:mix_entropy} is not written in terms of the conserved variable $\bu$ since it includes material quantities $\{\rho_k\}_{k=1}^{n_s}$ and $\{e_k\}_{k=1}^{n_s}$ which are found via the equation of state of each species. 
It can be shown that the mixture entropy in terms of the conserved variables is given by (\cite[Eqn.~16a]{renac2021entropy}): 
\begin{subequations} \label{eq:mixture_entropy}
\begin{align}
    &s(\bu)\eqq c_v(\bY) \log\Big( \frac{\rho e(\bu)}{\rho^{\gamma(\bY)}}\Big) + K(\bY), \text{ where }\\ 
    &K(\bY)\eqq \spsum Y_k \left(c_{v,k} \log\Big(\frac{c_{v,k}}{c_v(\bY)} \Big( \frac{r_k}{r(\bY)} \Big)^{\gamma_k - 1} \Big) + s_{\infty,k}\right),
\end{align}
\end{subequations}
Without loss of generality, we assume that $s_{\infty,k} = 0$ for all $k\in\{1:n_s\}$.

\begin{remark}[Mixture adiabatic index]
    In the literature, it is common to only report the species adiabatic index, $\{\gamma_k\}$, rather than the specific heat capacities, $\{c_{p,k}\}$ and $\{c_{v,k}\}$.
    Discrepancies in the choice of $\{c_{p,k}\}$ and $\{c_{v,k}\}$ can result in drastically different values of the mixture adiabatic index $\gamma(\bY)$.
    Consider the following valid but extreme example. 
    Let $\gamma_1 = 1.4$ and $\gamma_2 = 1.8$.
    Then one may choose, $c_{p,1} = 1.4$, $c_{v,1} = 1$, $c_{p,2} = 1800$, and $c_{v,2} = 1000$.
    Let $Y_1 = Y_2 = 1/2$. Then, the the mixture adiabatic index becomes, $\gamma(\bY) = 1801.4 / 1001 \approx 1.8$.
    Whereas, if $c_{p,2} = 1.8$ and $c_{v,2} = 1$, then $\gamma(\bY) = 1.6$.
\end{remark}

%
\begin{remark}[Dalton's Law and material quantities] \label{rem:mat_quantities}
    Note that the assumption of mechanical equilibrium is not necessary to recover the mixture pressure~\eqref{eq:pressure}. That is to say, only thermal equilibrium and Dalton's law's are sufficient (see: Proposition~\ref{prop:daltons_pressure}).
    However, thermal equilibrium and Dalton's law alone are not enough to recover the individual material densities $\{\rho_k\}_{k=1}^{n_s}$. 
    Assuming mechanical equilibrium, the material densities are then given by $\rho_k \eqq \frac{p}{(\gamma_k - 1) e_k}$ and the volume fractions are recovered by $\alpha_k \eqq \frac{\alpha_k\rho_k}{\rho_k}$.
\end{remark}
%

%

%
%
\subsection{Model properties}\label{sec:properties}
We now discuss properties of the model~\eqref{eq:model}. 
It was shown in~\citet{gouasmi2020minimum} that the total mixture entropy $-\rho s(\bu)$ is a convex functional and proved that a local minimum principle on the mixture specific entropy similar to that of \citet{tadmor1986minimum}.
We recall the result of \cite{gouasmi2020minimum}.
Let $\bu(\bx,t)$ be an entropy solution to \eqref{eq:model}. Then the following is satisfied:
\begin{equation} \label{eq:entropy_local_minimum}
    \essinf_{\Vert \bx \Vert_{\ell^2} \leq R} s(\bu(\bx,t)) \geq \essinf_{\Vert\bx\Vert_{\ell^2} \leq R + t v_{\max}} s(\bu(\bx,0)),
\end{equation}
where $R \geq 0$ is some radius and $v_{\max}$ denotes the maximum speed of information propagation.
Furthermore, it was shown that the system~\eqref{eq:model} admits an entropy pair.
\begin{definition}[Entropy solutions]
    We define the following entropy and entropy flux:
    \begin{equation}\label{eq:entropy_pair}
        \eta(\bu)\eqq -\rho s(\bu), \qquad \bq(\bu)\eqq -\bv\rho s(\bu) .
    \end{equation}
    Then the pair $(\eta(\bu), \bq(\bu))$ is an entropy pair for the system~\eqref{eq:model} weakly satisfying $\DIV\bq = (\GRAD_\bu \eta(\bu))\tr\DIV\polf(\bu)$. 
    The solution $\bu(\bx,t)$ is an entropy solution for~\eqref{eq:model} if it is a weak solution to the system and satisfies the following inequality in the weak sense: $\partial_t \eta(\bu) + \DIV\bq(\bu) \leq 0$.
\end{definition}
Following~\citet[Def.~2.3]{guermond2016invariant}, we now define the invariant set of the system.
\begin{definition}[Invariant set]
    The following convex domain is an invariant set for the multi-species model~\eqref{eq:model}:
    \begin{equation}\label{eq:invariant_set}
        \calA\eqq\left\{\bu\in\polR^{d+1+n_s} : \alpha_k \rho_k \geq 0, k\in\intset{1}{n_s}, \, e(\bu)>0, s(\bu)\geq s_{\min}\right\},
    \end{equation}
    where $s_{\min} \eqq \essinf_{\bx \in D} s(\bu_0(\bx))$.
\end{definition}
We see that \eqref{eq:entropy_local_minimum} implies the invariant set condition: $s(\bu) \geq s_{\min}$.

The goal of this work is to construct a first-order numerical method that is invariant-domain preserving with respect to~\eqref{eq:invariant_set} and satisfies discrete entropy inequalities. 
Furthermore, we would like to extend the method to second-order accuracy while remaining invariant-domain preserving. The starting point for the first-order methodology is the work of~\citep{guermond2016invariant}. It is shown in \citep[Sec.~4]{guermond2016invariant} if one can find the maximum wave speed to the Riemann problem for a hyperbolic system, then the first-order method proposed therein is invariant-domain preserving and satisfies discrete entropy inequalities. 

%% file: riemann_problem.tex
\section{The Riemann problem}\label{sec:riemann_problem}
In this section, we discuss the Riemann problem for the multi-species model~\eqref{eq:model}. 
For the sake of completeness, we give the full solution to the Riemann problem. To the best of our knowledge, the elementary wave structure for the model with two species was first analyzed in~\citet{larrouturou2006}. 
Furthermore, we present the maximum wave speed in the Riemann problem which is necessary for constructing a first-order invariant-domain preserving approximation that satisfies discrete entropy inequalities. To the best of our knowledge, the maximum wave speed for the multi-species model has never been explicitly derived. 

\subsection{Set up and summary}
Let $\bn$ be any unit vector in $\polR^d$. Then the one-dimensional Riemann problem for the multi-species model projected in the direction $\bn$ is given by:
\begin{equation} \label{eq:riemann_pb}
    \partial_t \bu + \partial_x(\polf(\bu) \bn) = \bm{0}, \quad  \bu(\bx, 0) = \begin{cases}
        \bu_L, \quad \text{ if } x < 0, \\
        \bu_R, \quad \text{ if } x > 0.
    \end{cases}
\end{equation}
Here $\bu\eqq( \alpha_1\rho_1, \ldots, \alpha_{n_s} \rho_{n_s}, m, E)^\mathsf{T}$ where $m\eqq\bbm\SCAL\bn$ is the one-dimensional momentum in the direction of $\bn$. Note that the one-dimensional velocity is defined by $v\eqq\bv\SCAL\bn= m/\rho$. 
The quantity $\bu_Z \eqq ((\alpha_1\rho_1)_Z, \ldots, (\alpha_{n_s} \rho_{n_s})_Z, m_Z, E_Z)^\mathsf{T}$, $Z \in \{L, R\}$, denotes either the left or right data. We assume that $\bu_Z\in\calA$.  
It is shown in~\citep[Sec.~2.3]{larrouturou2006} and in \S\ref{app:RP} that the elementary wave structure for~\eqref{eq:riemann_pb} consists of two genuinely non-linear waves (either expansion or shock) and a linearly degenerate contact wave with multiplicity $n_s$. 
A consequence of this wave structure is that the derivation of the maximum wave speed is the same for all $n_s \geq 2$ since the only modification would be an increase in multiplicity of the contact wave. 
Furthermore, it is shown that that the mass fractions $Y_k$, for all $k\in\intset{1}{n_s}$, are constant on each side of the contact wave (see Lemma~\ref{lem:mass_frac}). Consequently, $\gamma(\bY)$ is also constant on each side of the contact. This property is fundamental when deriving the maximum wave speed formula below. 
We now provide the novel formula.
\begin{proposition}[Maximum wave speed] \label{prop:wave_speeds}
Let $Z \in \{L, R\}$. Assume $p^{*}$ is a solution to $\varphi(p)=0$ where 
\begin{subequations}\label{eq:f_function}
\begin{align}
    &\varphi(p)\eqq f_L(p) + f_R(p) + v_R - v_L, \\
    &f_Z(p) = \begin{cases}
        (p - p_Z) \sqrt{\frac{A_Z}{p + B_Z}}, &\quad \text{ if } p > p_Z, \\ 
        \frac{2c_Z}{\gamma_Z - 1} \Big( \big( \frac{p}{p_Z} \big)^{\frac{\gamma_Z - 1}{2\gamma_Z}} - 1 \Big), &\quad \text{ if } p \leq p_Z.
    \end{cases}
\end{align}
\end{subequations}
and $A_Z = \frac{2}{(\gamma_Z + 1)\rho_Z}$, $B_Z = \frac{\gamma_Z - 1}{\gamma_Z + 1} p_Z$, $c_Z = \sqrt{\frac{\gamma_Z p_Z}{\rho_Z}}$.
Then the maximum wave speed is given by
\begin{equation} \label{eq:max_wave_speed}
    \lambda_{\max}(\bu_L, \bu_R, \bn_{LR}) \eqq \max(|\lambda_L(p^*)|, \lambda_R(p^*)|),
\end{equation}
where
\begin{align*}
    \lambda_L(p^*) \eqq v_L - c_L \sqrt{1 + \frac{\gamma_L + 1}{2\gamma_L}\max\Big( \frac{p^* - p_L}{p_L}, 0 \Big)}, \\
    \lambda_R(p^*) \eqq v_R + c_R \sqrt{1 + \frac{\gamma_R + 1}{2\gamma_R}\max\Big( \frac{p^* - p_R}{p_R}, 0 \Big)}.
\end{align*}
\end{proposition}

%% file: solution_RP.tex
\subsection{Elementary wave structure}\label{app:RP}
We now give a brief overview of the elementary wave structure.
To simplify the elementary wave analysis, we introduce the following mapping: $\bu \mapsto \btheta(\bu)$ where
\begin{equation}
    \btheta(\bu) \eqq \Big( \frac{\alpha_1\rho_1}{\rho(\bu)}, \ldots, \frac{\alpha_{n_s - 1} \rho_{n_s - 1}}{\rho(\bu)}, \rho(\bu), \frac{m}{\rho(\bu)}, E - \frac{m^2}{2\rho(\bu)} \Big)^\mathsf{T}.
\end{equation}
Using the mass fraction notation, we have that
$\btheta = (Y_1, \ldots, Y_{n_s-1}, \rho, v, \rho e)^{\mathsf{T}}$.
Note that the mapping $\bu \mapsto \btheta(\bu)$ is a smooth diffeomorphism. It can be shown that the eigenvalues of the Jacobian matrix, $D_\bu (\polf(\bu)\bn)$, correspond exactly to the eigenvalues of the matrix $\polB(\btheta) \eqq (D_\btheta \bu(\btheta))^{-1} D(\polf(\bu(\btheta))\bn)  D_{\btheta} \bu(\btheta)$ (see:~\cite[Chpt. II, Sec.~2.1.1]{godlewski2013numerical}).
That is, a smooth diffeomorphic change of variables does not affect the eigenvalues of the Jacobian matrix.
In short, the conservation law~\eqref{eq:riemann_pb} can be written as:
\begin{subequations}
\begin{align}
    \partial_t Y_i + v \partial_x Y_i &= 0, \quad \text{ for } i \in \intset{1}{n_s-1}, \\
    \partial_t \rho + v \partial_x \rho + \rho \partial_x v &= 0, \\
    \partial_t v + v \partial_x v + \rho^{-1} \partial_x p &= 0, \\
    \partial_t(\rho e) + v \partial_x (\rho e) + (\rho e + p) \partial_x v &= 0.
\end{align}
\end{subequations}
It can be shown that the transformed Jacobian matrix is:
\begin{equation}
    \polB(\btheta) = \begin{bmatrix}
        v \polI_{n_s-1} & \mathbf{0} & \mathbf{0} & \mathbf{0} \\
        \mathbf{0}\tr & v & \rho & 0 \\
        \rho^{-1}(D_{\mathbf{Y}} p)\tr & 0 & v & \frac{\gamma - 1}{\rho} \\
        \mathbf{0}\tr & 0 & \rho e + P & v
    \end{bmatrix},
\end{equation}
where $(D_\mathbf{Y} p)\tr = (\pdv{p}{Y_1}, \ldots, \pdv{p}{Y_{n_s-1}})$.
The eigenvalues are given by $\lambda_1(\btheta) = v - c$, $\lambda_{n_s + 2} = v + c$, and $\lambda_i = v$ for $i \in \intset{2}{n_s + 1}$, where $c = \sqrt{\gamma p / \rho}$.
The corresponding eigenvectors (as functions of $\btheta)$ are given by:
\begin{align}
    \br_1 = \begin{pmatrix}
        \mathbf{0}_{n_s-1} \\ \frac{\gamma - 1}{c} \\ -\frac{\gamma - 1}{\rho} \\ c
    \end{pmatrix},
    \,
    \br_i = \begin{pmatrix}
        \boldsymbol{\sfe}_{i-1} \\ 0 \\ 0 \\ 
        -\frac{1}{\gamma-1} \pdv{P}{Y_{i-1}}
    \end{pmatrix}, 
    \,
    \br_{n_s+1} = \begin{pmatrix}
        \mathbf{0}_{n_s-1} \\ 1 \\ 0 \\ 0
    \end{pmatrix},
    \,
    \br_{n_s+2} = \begin{pmatrix}
        \mathbf{0}_{n_s-1} \\ 
        \frac{\gamma-1}{c} \\ \frac{\gamma-1}{\rho} \\ c
    \end{pmatrix}
\end{align}
for $i \in \intset{2}{n_s}$ where $\{\boldsymbol{\sfe}_i\}_{i\in\intset{1}{n_s-1}}$ is the canonical basis for $\Real^{n_s-1}$.
\begin{lemma}
    The 1-wave and the $(n_s+2)$-wave are genuinely nonlinear and the $i$-waves for $i \in \intset{1}{n_s+1}$ are linearly degenerate.
\end{lemma}

\begin{proof}
    The wave structure is unaffected by the change of variables as shown in \cite[Chpt II, Sec. 2.1.1]{godlewski2013numerical}.
    The derivative of $\lambda_1$ is \begin{equation}
        D_\btheta \lambda_1(\btheta) = \begin{pmatrix}
            -\frac{(2\gamma - 1)c}{2\gamma(\gamma - 1)} \Big( \frac{\gamma_1 - \gamma}{c_v / c_{v,1}} \Big) \\
            \vdots \\
            -\frac{(2\gamma-1)c}{2\gamma(\gamma - 1)} \Big( \frac{\gamma_{n_s-1} - \gamma}{c_v / c_{v,{n_s-1}}} \Big) \\
            \frac{c}{2\rho} \\
            1 \\
            -\frac{c}{2\rho e}
        \end{pmatrix}.
    \end{equation}
    Thus, $D_\btheta \lambda_1(\btheta) \cdot \br_1(\btheta) = \frac{\gamma-1}{2\rho} - \frac{\gamma-1}{\rho} - \frac{\gamma(\gamma-1)}{2\rho} = -\frac{\gamma^2 - 1}{2\rho} < 0$.
    A similar result holds for $\lambda_{n_s+2}$.
    Then note that $D_{\btheta}\lambda_i(\btheta) = (\mathbf{0}_{n_s}, 1, 0)^{\sfT}$ for all $i \in \intset{2}{n_s+1}$, hence $D_\btheta\lambda_i(\btheta) \cdot \br_i(\btheta) = 0$.
    This completes the proof.
\end{proof}

\subsection{Solution to the Riemann problem}
We now give the full solution to the Riemann problem.
We first show that the mass fractions are constant across each nonlinear wave.
\begin{lemma}[Mass fractions] \label{lem:mass_frac}
    Let $Y_k(x,t)$ be the mass fraction weak solution to the Riemann problem \eqref{eq:riemann_pb} for each $k \in \intset{1}{n_s}$.
    Then $Y_k(x,t) = Y_{k,L}$ for all $x < v^*t$ and $Y_k(x,t) = Y_{k,R}$ for all $x > v^*t$ for all $t > 0$ where $v^*$ denotes the speed of the contact.
\end{lemma}
\begin{proof}
Assume that the left state, $L$, is connected to the state across the 1-wave by a shock wave.
Then the multi-species system satisfies the Rankine-Hugoniot relations: $S_L(\bu_L - \bu_{*L}) = (\polf(\bu_L) - \polf(\bu_{*L}))\bn$.
In particular, $S_L(\rho_L Y_{i,L} - \rho_{*L} Y_{k,*L}) = (\rho_L v_L Y_{k,L} - \rho_{*L} v_{*L} Y_{k,*L})$.
This identity can be rewritten as, $\rho_L Y_{k,L}(S_L - v_L) = \rho_{*L} Y_{k,*L}(S_L - v_{*L})$.
From the conservation of mass, we have $\rho_L(S_L - v_L) = \rho_{*L} (S_L - v_{*L})$ and therefore we conclude that $Y_{k,L} = Y_{k,*L}$ for all $k \in \intset{1}{n_s}$.

Assume now that the left state is connected across a 1-wave by an expansion wave.
Note that $Y_k$ satisfies $\frac{D Y_k}{Dt} := \partial_t Y_k + v_k \partial_x Y_k = 0$ hence $Y_k$ is constant across an expansion wave.
This same reasoning can be applied across the right wave.
\end{proof}

Since each $Y_k$ is constant across the left and right waves, we conclude that $\gamma(\bY)$ is constant across the left and right waves.
Therefore, left of the contact ($x/t < v_*$), the pressure law obeys $p = (\gamma_L - 1) \rho e$.
Across the right wave the pressure law is given by $p = (\gamma_R - 1)\rho e$.
Furthermore, from the Gibbs mixture identity~\eqref{eq:gibbs_mixture} the differential relationship on the left and right waves is:
\begin{equation} \label{eq:gibbs_identity_on_wave}
    T\diff s = \diff e + p \diff \tau.
\end{equation}
Hence, across the genuinely nonlinear waves, the pressure is $p = -\partial_\tau e(\tau, s)$.
Furthermore, it can be shown that the specific entropy, $s$, satisfies $\partial_t s + \bv \cdot \nabla s = 0$ as indicated in \cite[Sec.~2]{gouasmi2020minimum}.
Hence, $s$ is constant across expansions.
From \eqref{eq:mixture_entropy}, the specific internal energy as a function of $s$, $\tau$, and $\bY$ is,
\begin{equation}
    e(\bu) \eqq \tau(\bu)^{-(\gamma(\bY) - 1)} \exp\Big( \frac{s(\bu) - K(\bY)}{c_v(\bY)} \Big).
\end{equation}
From the differential relation \eqref{eq:gibbs_identity_on_wave}, the pressure on an expansion wave (as a function of $\rho$) is,
\begin{equation}
    p = (\gamma(\bY_Z) - 1) \rho^{\gamma(\bY_Z)}\exp\Big(\frac{1}{c_v(\bY_Z)}\big(s_Z - K(\bY_Z)\big)\Big) = C_Z \rho^{\gamma_Z},
\end{equation}
for $Z \in \{L, R\}$.
The constant $C_Z$, can be computed simply by $C_Z = p_Z / \rho_Z^{\gamma_Z}$ since the expansion waves begins from the constant state.
We also have the same structure for the Riemann invariants as in the single material case, in particular, the $1$-Riemann invariant is $w_1(\btheta) = v + \frac{2c}{\gamma - 1}$ and the $(n_s+2)$-Riemann invariant is $w_{n_s+2}(\btheta) = v - \frac{2c}{\gamma-1}$.

Since $\gamma = \gamma_L$ to the left of the contact and $\gamma = \gamma_R$ to the right of the contact, the pressure law is given by $p = (\gamma_Z - 1) \rho e$ for $Z \in \{L, R\}$ on each side of the contact, respectively.
The solution to this Riemann problem can be computed exactly as done in \citet[Sec.~4]{clayton2022invariant} for $\rho$, $p$, and $v$.
(A ``two-gamma'' approximation was used in~\cite{clayton2022invariant} for interpolating general equations of state, inspired by \citet{abgrall2001computations}; the connection to the multi-species model~\eqref{eq:model} was not yet made at the time of the publication.)
Once $\rho$, $p$, and $v$ are known, the partial densities are computed by $\alpha_k \rho_k = \rho Y_k$, since the mass fractions, $\{Y_k\}$, are piecewise constant.
The fundamental methodology of constructing the solution can also be found in \citet[Sec.~4]{toro2013riemann}, \citet[Chpt.~III, Sec.~3]{godlewski2013numerical}, and \citet{Lax_1957}.
For brevity, we simply present the result.

The pressure in the star domain is determined by solving the following nonlinear equation:
\begin{equation} \label{eq:varphi=0}
    \varphi(p) \eqq f_L(p) + f_R(p) + v_R - v_L = 0,
\end{equation}
where 
\begin{equation}
    f_Z(p) = \begin{cases}
        (p - p_Z) \sqrt{\frac{A_Z}{p + B_Z}}, &\quad \text{ if } p > p_Z, \\ 
        \frac{2c_Z}{\gamma_Z - 1} \Big( \big( \frac{p}{p_Z} \big)^{\frac{\gamma_Z - 1}{2\gamma_Z}} - 1 \Big), &\quad \text{ if } p \leq p_Z,
    \end{cases}
\end{equation}
for $Z \in \{L, R\}$ and where $A_Z = \frac{2}{(\gamma_Z + 1)\rho_Z}$, $B_Z = \frac{\gamma_Z - 1}{\gamma_Z + 1} p_Z$, and $c_Z = \sqrt{\frac{\gamma_Z p_Z}{\rho_Z}}$.
The details for this derivation can be found in  and are independent of the number of materials present.
Let $\bc(x,t) \eqq (\rho(x,t), v(x,t), p(x,t))^{\sfT}$ be the primitive state of the solution to the Riemann problem for the multi-component Euler equations.
Define the self-similar parameter $\xi \eqq \frac{x}{t}$.
Then the weak entropy solution to the Riemann problem is given by:
\begin{equation} \label{eq:riemann_solution}
    \bc(x,t) \eqq \begin{cases}
        \bc_L, &\quad \text{ if } \xi < \lambda_L^{-}(p^*), \\
        \bc_{LL}(\xi), &\quad \text{ if } \lambda_L^{-}(p^*) \leq \xi < \lambda^+_L(p^*), \\
        \bc^*_L, &\quad \text{ if } \lambda^+_L(p^*) \leq \xi < v^*, \\
        \bc^*_R, &\quad \text{ if } v^* \leq \xi < \lambda^-_R(p^*) \\ \bc_{RR}(\xi), &\quad \text{ if } \lambda_R^{-}(p^*) \leq \xi < \lambda^+_R(p^*), \\
        \bc_R, &\quad \text{ if } \lambda^+_R(p^*) \leq \xi
    \end{cases}
\end{equation}
where 
\begin{subequations}
\begin{align}
    \bc_{LL}(\xi) &= \Big( \rho_L \Big( \frac{2}{\gamma_L + 1} + \frac{(\gamma_L - 1)(v_L - \xi)}{(\gamma_L + 1) c_L} \Big)^{\frac{2}{\gamma_L - 1}}, v_L - f_L(p(\xi)), p_L(\xi)\Big)\tr \\
    \bc_L^* &= \begin{cases}
        \bc_{LL}(\lambda^+_L), \quad &\text{ if } p^* < p_L \\
        ( \rho^*_L, v^*, p^* )\tr \quad &\text{ if } p^* \geq p_L \\
    \end{cases} \\
    \bc_{RR}(\xi) &= \Big( \rho_R \Big( \frac{2}{\gamma_R + 1} - \frac{(\gamma_R - 1)(v_R - \xi)}{(\gamma_R + 1) c_R} \Big)^{\frac{2}{\gamma_R - 1}}, v_R + f_R(p(\xi)), p_R(\xi)\Big)\tr \\
    \bc_R^* &= \begin{cases}
        \bc_{RR}(\lambda^-_R), \quad &\text{ if } p^* < p_R \\
        ( \rho^*_R, v^*, p^* )\tr \quad &\text{ if } p^* \geq p_R \\
    \end{cases}
\end{align}
\end{subequations}
where $\rho_L(\xi) = \rho_L\big( \frac{2}{\gamma_L + 1} + \frac{\gamma_L - 1}{(\gamma_L + 1) c_L}(v_L - \xi) \big)^{\frac{2}{\gamma_L - 1}}$, 
$\rho_R(\xi) = \rho_R\big( \frac{2}{\gamma_R + 1} - \frac{\gamma_R - 1}{(\gamma_R + 1) c_R}(v_R - \xi) \big)^{\frac{2}{\gamma_R - 1}}$,
$p_L(\xi) \eqq C_L \rho_L(\xi)^{\gamma_L}$, and $p_R(\xi) \eqq C_R \rho_R(\xi)^{\gamma_R}$.
Across a shock, the density is:
\begin{equation}
    \rho^*_Z = \frac{\rho_Z \big( \frac{p^*}{p_Z} + \frac{\gamma_Z - 1}{\gamma_Z + 1} \big)}{\frac{\gamma_Z - 1}{\gamma_Z + 1} \frac{p^*}{p_Z} + 1},
\end{equation}
for $Z \in \{L, R\}$ where $p^*$ solves \eqref{eq:varphi=0} and $v^* = v_L - f_L(p^*) = v_R + f_R(p^*)$.
The wave speeds are given by:
\begin{align*}
    \lambda^-_L(p^*) \eqq v_L - c_L \sqrt{1 + \frac{\gamma_L + 1}{2\gamma_L}\max\Big( \frac{p^* - p_L}{p_L}, 0 \Big)}, \\
    \lambda^+_L(p^*) \eqq \begin{cases}
        v_L - f_L(p^*) - c_L \big( \frac{p^*}{p_L} \big)^{\frac{\gamma_L - 1}{2\gamma_L}}, &\quad \text{ if } p^* < p_L \\
        \lambda^-_L(p^*), &\quad \text{ if } p^* \geq p_L,
    \end{cases} \\
    \lambda^+_R(p^*) \eqq v_R + c_R \sqrt{1 + \frac{\gamma_R + 1}{2\gamma_R}\max\Big( \frac{p^* - p_R}{p_R}, 0 \Big)}, \\
    \lambda^-_R(p^*) \eqq \begin{cases}
        v_R + f_R(p^*) + c_R \big( \frac{p^*}{p_R} \big)^{\frac{\gamma_R - 1}{2\gamma_R}}, &\quad \text{ if } p^* < p_R, \\
        \lambda^+_R(p^*), &\quad \text{ if } p^* \geq p_R.
    \end{cases}
\end{align*}
From \citet{clayton2022invariant}, the waves are well ordered from the following lemma.
\begin{lemma}
For $\gamma_L, \gamma_R > 1$ and $c_L, c_R > 0$, we have that,
\begin{equation*}
    \lambda^-_L(p^*) \leq \lambda^+_L(p^*) \leq v^*_L \leq v^*_R \leq \lambda^-_R(p^*) \leq \lambda^+_R(p^*).
\end{equation*}
\end{lemma}

We now present an essential result necessary for constructing the numerical method described in Section~\ref{sec:approximation}.
\begin{lemma}[Minimum entropy in the Riemann solution] \label{lem:minimum_entropy_riemann}
    Let $\bu(x, t)$ be the weak solution to the Riemann problem \eqref{eq:riemann_pb} defined by \eqref{eq:riemann_solution}.
    Let $\widehat{\lambda}_{\max}$ denote an upper bound on the maximum wave speed. Let the average of the Riemann solution be given by: $\overline{\bu}(t) \eqq \frac{1}{2\widehat{\lambda}_{\max} t}\int_{-\widehat{\lambda}_{\max}t}^{\widehat{\lambda}_{\max}t} \bu(x, t) \diff x$.
    Then $\overline{\bu}(t)$ satisfies:
    \begin{equation}
        s(\overline{\bu}(t)) \geq \min(s(\bu_L), s(\bu_R)).
    \end{equation} 
\end{lemma}

\begin{proof}
    Since $(\rho s)(\bu)$ is concave, we can apply Jensen's inequality:
    \begin{equation}
        (\rho s)(\overline{\bu}(t)) \geq \frac{1}{2\widehat{\lambda}_{\max} t} \int_{-\wlambda_{\max} t}^{\wlambda_{\max} t} (\rho s)(\bu(x,t)) \diff x.
    \end{equation}
    Consider the case that the $L$-wave is an expansion wave.
    Then the entropy is constant up to the contact: $v^* = x/t$.
    That is, $s(\bu(x,t)) = s_L$ for all $x < v^* t$.
    In the case, that the left wave is a shock, one has that the specific entropy must increase.
    This can be seen by noting that left of the contact, $x < v^* t$, the equation of state behaves as a single material equation of state due to Lemma~\ref{lem:mass_frac}.
    As such, since the shock is compressive, the entropy must increase
    (See:~\citet[Chpt.~III, Sec.~2]{godlewski2013numerical} for more details).
    Therefore, $s(\bu(x,t)) \geq s_L$ for all $x < v^*t$.
    The same reasoning can be applied across the right wave.
    Therefore, $s(\bu(x,t)) \geq \min(s_L, s_R)$, pointwise a.e. for all $t > 0$. 
    Therefore, $(\rho s)(\overline{\bu}(t)) \geq \frac{\min(s_L, s_R)}{2\wlambda_{\max} t} \int_{-\wlambda_{\max}t}^{\wlambda_{\max} t} \rho(x,t) \diff x = \overline{\rho}(t) \min(s_L, s_R)$.
    But $(\rho s)(\overline{\bu}(t)) = \overline{\rho}(t) s(\overline{\bu}(t))$, hence the result follows.
\end{proof}

%% file: spatial.tex
\section{Approximation details}\label{sec:approximation}
The spatial approximation we adapt in this paper is based on the invariant-domain preserving methodology introduced in \citet{guermond2016invariant}. The low-order method can be thought of as a discretization-independent generalization of the algorithm proposed in~\citet[pg.~163]{lax1954}. Various extensions for the compressible Euler equations have been proposed in~\citep{guermond2018second, clayton2022invariant, clayton2023robust, clayton2025approximation}.  For the sake of brevity, we omit the full approximation details and refer the reader to the previous references.

\subsection{Low-order method}\label{sec:low_order}
We now introduce the low-order approximation. Let $\calV$ denote the index set enumerating the degrees of freedom. Let $\calI(i)$ denote an index set for the local stencil for the degree of freedom, $i$.
Then, for every $i\in\calV$ and $j\in\calI(i)$, the low-order method with forward Euler time-stepping is given by:
\begin{equation}\label{eq:low_order}
    \frac{m_i}{\tau}(\bsfU_i\upLnp - \bsfU_i\upn) = \ssum\left[ -\left(\polf(\bsfU_j^n)-\polf(\bsfU_i^n)\right)\cij + \dijL (\bsfU_j^n - \bsfU_i^n) \right],
\end{equation}
where
\begin{equation} \label{eq:dijL}
    \dijL\eqq \max(\wlambda_{\max}(\bsfU_i^n, \bsfU_j^n,\bn_{ij})\|\cij\|, \wlambda_{\max}(\bsfU_j^n, \bsfU_i^n,\bn_{ji})\|\cji\|)
\end{equation}
and $\wlambda_{\max}\geq\lambda_{\max}$ is a suitable upper bound on the maximum wave speed in the local Riemann problem for $(\bsfU_i^n, \bsfU_j^n, \bn_{ij})$.
It is shown in \cite[Rem.~3.1]{guermond2016invariant} that the method~\eqref{eq:low_order} is globally mass conservative; that is to say $\sum_{i\in\calV}m_i\bsfU_i\upLnp  = \sum_{i\in\calV}m_i\bsfU_i\upn$. When using linear finite elements as the underlying discretization, the method is formally first-order accurate in space~\citep{guermond2016invariant}.

As shown in Proposition~\ref{prop:wave_speeds}, the computation of the discrete local maximum wave speed, $\lambda_{\max}$, would require the solution to a nonlinear equation for every $(i,j)$ pair which can be quite costly. Instead, we opt to use an upper bound on the maximum wave speed, $\widehat{\lambda}_{\max}$, which will be more efficient to compute.
As referenced in Section~\ref{sec:riemann_problem}, the maximum wave speed can be found in~\citet{clayton2022invariant} as well as an algorithm for computing the upper bound~\cite[Alg.~1]{clayton2022invariant} which we use in this work.
\begin{theorem}[Invariant-domain preserving]
    The low order method in \eqref{eq:low_order} and \eqref{eq:dijL} using the upper bound on the maximum wave speed, $\widehat{\lambda}_{\max}$, described in \cite[Alg.~1]{clayton2022invariant} under the CFL condition $1 + \frac{2\tau d_{ii}\upLnp}{m_i}\geq 0$ is invariant-domain preserving.
    That is, $\bsfU_i\upLnp \in \calA$ for all $i \in \calV$.
    Furthermore, the update $\bsfU_i\upLnp$ satisfies discrete entropy inequalities.
\end{theorem}
\begin{proof}
    The proof follows directly by the application of Theorem 4.1 and Theorem 4.7 in~\citet{guermond2016invariant}.
\end{proof}

\subsection{Local bounds}
An important and well-known property of the method~\eqref{eq:low_order} is that it can be written as a convex combination of  ``bar states'' under the CFL condition $1 + \frac{2\tau d_{ii}\upLnp}{m_i}\geq 0$:
\begin{subequations}
    \begin{align}
         & \bsfU_i\upLnp = \left(1 + \frac{2\tau d_{ii}\upLnp}{m_i}\right)\bsfU_i^n + \snsum\frac{2 \tau \dijL}{m_i}\Ubar_{ij}\upn, \text{ where } \label{eq:bar_state_update} \\
         & \Ubar_{ij}\upn = \frac12(\bsfU_i\upn + \bsfU_j\upn) - \frac{1}{2 \dij\upLn}(\polf(\bsfU_j\upn) - \polf(\bsfU_i\upn))\cij\label{eq:bar_states}.
    \end{align}
\end{subequations}
When $\dijL$ is defined by~\eqref{eq:dijL}, it can be shown that $\Ubar\upn_{ij}\in\calA$ (see:~\citep[Thm~4.1]{guermond2016invariant}).
\begin{remark}[Bar states] \label{rem:bar_states}
    An important result regarding the bar states~\eqref{eq:bar_states} $\{\overline{\bsfU}^n_{ij}\}$ is that they are the average of the discrete Riemann solution.
    That is, $\overline{\bsfU}^n_{ij} = \overline{\bu}(t)$, where $\overline{\bu}(t)$ is the average of the Riemann solution for the state $(\bsfU^n_i, \bsfU^n_j, \bn_{ij})$ at the time $t = \frac{\Vert \bc_{ij} \Vert_{\ell^2}}{2\dijL}$.
    This is a classical result (see:  \cite[Lemma~2.1]{guermond2016invariant}).
\end{remark}
Notice that~\eqref{eq:bar_state_update} is a \textit{convex} combination of the states $\{\overline{\bsfU}^n_{ij}\}$ and therefore satisfies local bounds in space and time.
More specifically, we have that if $\bsfU_i\upLnp\in\calA$ for all $i\in\calV$, then $\Psi(\bsfU_i\upLnp) \geq \min_{j\in\calI(i)} \Psi(\Ubar_{ij}\upn)$ where $\Psi(\bu)$ is any quasiconcave functional. This fact will be used in the convex limiting section \S\ref{sec:lim_concave}.

\subsection{Provisional high-order method}\label{sec:high_order}
We now present a provisional high-order method with forward Euler time-stepping. The method follows that of~\citep[Eqn.~3.1]{clayton2023robust} where the modification here is in how the ``entropy indicator'' is defined (see: \S\ref{sec:entropy_indicator}).
For every $i\in\calV$ and $j\in\calI(i)$, we define the higher-order update:
\begin{subequations}
    \begin{align}
         & \frac{m_i}{\tau}(\bsfU_i\upHnp - \bsfU_i^n) = \ssum\left[\bsfF_{ij}\upHn + b_{ij} \bsfF_{j}\upHn - b_{ji}\bsfF_{i}\upHn\right]\, \text{ with } \\
         & \bsfF_{ij}\upHn\eqq -(\polf(\bsfU_j)-\polf(\bsfU_i^n))\cij + \dijH (\bsfU_j^n - \bsfU_i^n), \quad \bsfF_i\upHn\eqq\ssum\bsfF_{ij}\upHn.
    \end{align}
\end{subequations}
Here, two modifications have been made to the low-order method~\eqref{eq:low_order} to achieve higher-order accuracy in space.
\textup{(i)} We replaced the lumped mass matrix by an approximation of the consistent mass matrix to reduce dispersive errors.
That is to say, with $\sfX\in\polR^I$, where $I\eqq\text{card}(\calV)$, we have $(\polM^{-1} \sfX)_i\approx \sfX_i + \ssum(b_{ij}\sfX_j - b_{ji}\sfX_i)$ where $b_{ij}\eqq \delta_{ij}-\frac{m_{ij}}{m_j}$ and $\delta_{ij}$ denoting the Kronecker symbol.
\textup{(ii)} We replaced the low-order graph-viscosity coefficient by $d_{ij}\upHn\eqq\frac12(\zeta_i^n + \zeta^n_j) \cdot d_{ij}\upLn$ where  $\zeta\upn_i\in[0,1]$ is an indicator for entropy production and scales like $\calO(h)$ for piecewise linear finite elements where $h$ is the typical mesh size.

\subsection{Entropy indicator}\label{sec:entropy_indicator}
We now introduce an entropy indicator which is inspired by~\citep[Sec.~3.2]{clayton2023robust}.
The idea is as follows. For every $i$ and at every $t^n$, we consider a surrogate evolution of the full mixture: $\partial_t\bw + \DIV\polf^{i,n}(\bw) = \bm{0}$ where $\bw\eqq(\rho,\bbm, E)^{\mathsf{T}}$ and
\begin{equation}\label{eq:surrogate_flux}
    \polf^{i,n}(\bw)\eqq\begin{pmatrix}
        \bbm,                                   \\
        \bv\otimes\bbm + \tp^{i,n}(\bw)\polI_d, \\
        \bv(E + \tp^{i,n}(\bw))
    \end{pmatrix},
    \qquad
    \tp^{i,n}(\bw)\eqq(\gamma_i^{\min,n} - 1)\rho e(\bw),
\end{equation}
where $\gamma_i^{\min,n} \eqq \min_{j\in\calI(i)} \gamma(\bY^n_j)$.
Note that we have slightly abused notation by introducing $e$ as a function of $\bw$; however, we emphasize that $e(\bw) = e(\bu) = \rho^{-1} E - \frac12 \Vert \bv \Vert^2_{\ell^2}$.
We further define the respective ``surrogate entropy pair'' for the flux~\eqref{eq:surrogate_flux} by:
\begin{subequations}
    \begin{align}
         & \eta^{i,n}(\bw)\eqq\left(\rho^2e(\bw)\right)^{\frac{1}{\gamma_i^{\min,n}+1}} - \frac{\rho}{\rho_i^n}\left((\rho_i^n)^2e(\bsfW_i^n)\right)^{\frac{1}{\gamma_i^{\min,n}+1}}, \\
         & \bF^{i,n}(\bw)\eqq\bv\eta^{i,n}(\bw).
    \end{align}
\end{subequations}
Here, $\bsfW_i^n\eqq (\rho_i\upn, \bbm(\bsfU_i\upn), E(\bsfU_i\upn))^{\mathsf{T}}$ where $\rho_i\upn\eqq\spsum\alpha_k\rho_k(\bsfU_i\upn)$.
The idea now is to measure a discrete counterpart to: $\DIV\bF^{i,n}(\bw) = (\GRAD_{\bw}\eta^{i,n}(\bw))^{\mathsf{T}}\DIV\polf^{i,n}(\bw)$ which can be thought as an estimate to ``entropy production''. This is done via the entropy indicator $\zeta_i^n$ defined by:
\begin{subequations}
    \begin{align}
         & \zeta_i^n\eqq \frac{\abs{N_i^n}}{D_i^n + \frac{m_i}{\abs{D}}\eta^{i,n}(\bsfW_i^n)},                                        \\
         & N_i^n\eqq \ssum\left[\bF^{i,n}(\bsfW_j^n)\cij - (\GRAD_\bu\eta^{i,n}(\bsfW_i^n)\tr(\polf^{i,n}(\bsfW_j^n)\cij)\right],     \\
         & D_i^n\eqq \abs{\ssum\bF^{i,n}(\bsfW_j^n)\cij} + \ssum \abs{(\GRAD_\bu\eta^{i,n}(\bw_i^n))\tr(\polf^{i,n}(\bsfW_j^n)\cij)},
    \end{align}
\end{subequations}
where $m_i$ is the respective mass associated with the degree of freedom $i$ and $\abs{D}$ is the measure of the spatial domain.

We illustrate the performance of the entropy indicator in Figure~\ref{fig:entropy_indicator} with a two-species extension of the standard Woodward-Colella blast wave benchmark~\citep{woodward1984numerical} using $3201$ $\polQ_1$ degrees of freedom.
We assume that the high-pressure regions contain only air ($\gamma_1 = \frac{1005}{718}$) and assume the low-pressure region contains only helium ($\gamma_2 = \frac{5193}{3115}$).
The results are presented for the time snapshots $t = \{\mynum{0.015}{s}, \mynum{0.038}{s}\}$. We see that the entropy indicator (deep red) is almost zero everywhere except at the discontinuities of the mixture density (black). We further see that at the mass fraction $Y_1$ (teal) discontinuities, the entropy indicator is small which implies the method is near optimal at the species interface.

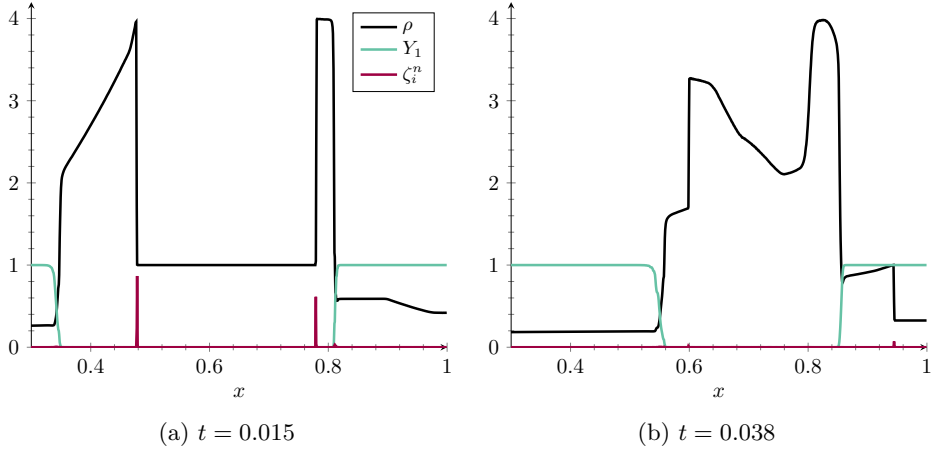
\begin{figure}[htbp!]
    \centering
    \subfloat[$t = 0.015$]{\adjustbox{width=0.49\linewidth, valign=b}{\input{FIGS/wwc_t15}}}
    \subfloat[$t = 0.038$]{\adjustbox{width=0.49\linewidth, valign=b}{\input{FIGS/wwc_t38}}}
    \caption{Entropy indicator illustration with multi-species Woodward-Colella blast wave. }\label{fig:entropy_indicator}
\end{figure}

%% file: FIGS/wwc_t15.tex
\begin{tikzpicture}[spy using outlines={rectangle, height=3cm,width=2.3cm, magnification=3, connect spies}]
	\begin{axis}[name=plot1,
		axis line style={latex-latex},
	    axis x line=left,
        axis y line=left,
		xlabel={$x$},
    	xmin=0.3, xmax=1,
    	xtick={0.4, 0.6, 0.8, 1.0},
    	ymin=0,ymax=4.2,
    	ytick={0, 1, 2, 3, 4},
        minor x tick num=4,
        minor y tick num=4,
        clip mode=individual,
    	ylabel style={rotate=-90},
    	legend style={at={(0.97, 0.97)},anchor=north east,font=\small},
    	legend cell align={left},
    	style={font=\normalsize}]

        \addplot[color=PlotColor5, 
                style={very thick}]
        table[x=x, y=rho, col sep=comma]{./data/wwc-t15-ds.csv};
        \addlegendentry{$\rho$};
        
        \addplot[color=PlotColor3, 
                style={very thick}]
        table[x=x, y=Y_0, col sep=comma]{./data/wwc-t15-ds.csv};
        \addlegendentry{$Y_1$};        
                
        \addplot[color=PlotColor1, 
                style={very thick}]
        table[x=x, y=alpha, col sep=comma]{./data/wwc-t15-ds.csv};
        \addlegendentry{$\zeta_i^n$};

	\end{axis}
\end{tikzpicture}

%% file: FIGS/wwc_t38.tex
\begin{tikzpicture}[spy using outlines={rectangle, height=3cm,width=2.3cm, magnification=3, connect spies}]
	\begin{axis}[name=plot1,
		axis line style={latex-latex},
	    axis x line=left,
        axis y line=left,
		xlabel={$x$},
    	xmin=0.3, xmax=1,
    	xtick={0.4, 0.6, 0.8, 1.0},
    	ymin=0,ymax=4.2,
    	ytick={0, 1, 2, 3, 4},
        minor x tick num=4,
        minor y tick num=4,
        clip mode=individual,
    	ylabel style={rotate=-90},
    	legend style={at={(0.97, 0.97)},anchor=north east,font=\small},
    	legend cell align={left},
    	style={font=\normalsize}]
    	
        \addplot[color=PlotColor5, 
                style={very thick}]
        table[x=x, y=rho, col sep=comma]{./data/wwc-t38-ds.csv};
        
        \addplot[color=PlotColor3, 
                style={very thick}]
        table[x=x, y=Y_0, col sep=comma]{./data/wwc-t38-ds.csv};
                
        \addplot[color=PlotColor1, 
                style={very thick}]
        table[x=x, y=alpha, col sep=comma]{./data/wwc-t38-ds.csv};

	\end{axis}
\end{tikzpicture}

%% file: limiting.tex
\section{Convex limiting}\label{sec:limiting}
It is discussed in \citep{clayton2023robust} that the provisional high-order update $\bsfU\upHnp$ defined in the previous section is not guaranteed to be invariant-domain preserving. In this section, we present a convex limiting technique that corrects this issue. The novelty of the approach in this paper is the limiting procedure for enforcing the minimum principle on general concave functionals described in Section~\ref{sec:lim_concave}. This procedure will be used to enforce: \textup{i)} a local maximum and minimum principle on the partial densities, \textup{ii)} the positivity of the mixture internal energy and \textup{iii)} the minimum principle on the mixture entropy

\subsection{Set up}
The methodology is loosely based on the Flux-Corrected Transport methodology (see:~\cite{zalesak1979fully, boris1997flux, kuzmin2012flux}) and follows directly the works of~\citep{guermond2018second,clayton2023robust}. The limited update is given by:
\begin{subequations}\label{eq:high_order_update}
\begin{align}
\bsfU_i\upnp &= \sum_{j\in\calI(i)\setminus\{i\}} \omega_i\left(\bsfU_i\upLnp + \ell_{ij}\upn\bsfP_{ij}^n\right), \label{eq:limited}\\ 
\bsfP_{ij}^n &= \frac{\tau}{m_i \omega_i}\left(\bsfF_{ij}\upHn-\bsfF_{ij}\upLn + b_{ij}\bsfF_{j}\upHn - b_{ji}\bsfF_{i}\upHn\right).
\end{align}
\end{subequations} 
where the limiter coefficient is such that $\ell\upn_{ij}\in[0,1]$ and is defined to be symmetric $\ell\upn_{ij} = \ell\upn_{ji}$.
The weights $\omega_i$ form a set of convex coefficients and are defined by $\omega_i\eqq\frac{1}{\text{card}(\calI(i)\setminus\{i\})}$. Note that when $\ell_{ij}\upn = 0$, the update~\eqref{eq:limited} reduces to $\bsfU_i\upnp = \bsfU_i\upLnp$. Similarly, when $\ell_{ij}\upn = 1$, the update~\eqref{eq:limited} reduces to $\bsfU_i\upnp = \bsfU_i\upHnp$. 
Note that for each $i\in\calV$, the update~\eqref{eq:limited} is a convex combination of the states $\bsfU_i\upLnp + \ell_{ij}\upn\bsfP_{ij}\upn$ for all $j\in\calI(i)\setminus\{i\}$. Thus, if we can find an $\ell_{ij}\upn$ for each pair $(i,j)$ such that $\bsfU_i\upLnp + \ell_{ij}\upn\bsfP_{ij}\upn\in\calA$, then the update~\eqref{eq:limited} will be a convex combination of invariant-domain preserving states and thus invariant-domain preserving itself.
We now present a general algorithm for finding the optimal limiter coefficient such that the limited updated is invariant-domain preserving.

\subsection{General limiting on concave functionals} \label{sec:lim_concave}
In this section, we simplify the limiting process described in~\citep{guermond2019invariant} by using a linear interpolation between the low-order update and high-order update.
The method is only slightly more restrictive as it requires the functional to be concave rather than quasiconcave as a function of the conserved variable $\bu$.
We note that the partial densities are trivially concave and the internal energy is concave (see \cite[Sec.~4.1]{guermond2018second}).
Furthermore, it was shown in \citet[Sec.~2]{gouasmi2020minimum}, that $\rho s(\bu)$ is also concave. Thus, the constraints of interest for the multi-species model~\eqref{eq:model} will all be concave.
More specifically, we define: 
\begin{subequations} \label{eq:concave_functionals}
    \begin{align}
        \Psi_i^k(\bu) &\eqq (\alpha_k\rho_k)(\bu) - (\alpha_k\rho_k)_{i}^{\min, n}, \\
        \Psi_i^{n_s+k}(\bu) &\eqq (\alpha_k\rho_k)_{i}^{\max, n} - (\alpha_k\rho_k)(\bu), \\
         \Psi_i^{2n_s+1}(\bu) &\eqq \varepsilon(\bu) - \varepsilon_i^{\min, n}, \\
         \Psi_i^{2n_s+2}(\bu) &\eqq \sigma(\bu) - \sigma_i^{\min, n},
    \end{align}
\end{subequations}
for $k \in \intset{1}{n_s}$ where $\sigma(\bu)\eqq\rho s(\bu)$. 
The local bounds are defined as follows:
\begin{subequations}
\begin{align}
    &(\alpha_k\rho_k)^{\min,n}_i \eqq \min_{j\in\calI(i)} (\overline{(\alpha_k\rho_k)}^n_{ij}, (\alpha_k \rho_k)_j), 
    &&\varepsilon^{\min}_i \eqq \min_{j\in\calI(i)} (\varepsilon(\overline{\bsfU}^n_{ij}), \varepsilon(\bsfU^n_j)) \\
    &(\alpha_k\rho_k)^{\max,n}_i \eqq \max_{j\in\calI(i)} (\overline{(\alpha_k\rho_k)}^n_{ij}, (\alpha_k \rho_k)_j),
    &&\sigma^{\min}_i \eqq \min_{j\in\calI(i)} \sigma(\bsfU^n_j),
\end{align}
\end{subequations}
It was shown in \S\ref{sec:low_order} that the low-order update satisfies $\Psi^{\nu}(\bsfU_i\upLnp)\geq 0$ for all $i\in\calV$ and every $\nu$ in the \textit{ordered} set $\intset{1}{2n_s+2}$.
From Lemma~\ref{lem:minimum_entropy_riemann} and Remark~\ref{rem:bar_states}, we see that the expected discrete minimum entropy principle is encoded in the inequality $\Psi_i^{2n_s+2}(\bsfU_i\upLnp)\geq 0$.
\begin{remark}[Locally invariant-domain preserving]
Note that the above bounds can be used to define a \textit{local} invariant set since $\Psi^{\nu}(\bsfU_i\upLnp)\geq 0$ for all $i\in\calV$. That is, for each degree of freedom $i\in\calV$ we define:
\begin{equation}
    \calB_i \eqq \bigcap_{\nu = 1}^{2n_s+2} \calB^\nu_i,
\end{equation}
where
\begin{equation}
  \calB^\nu_i \eqq \{\bsfU \in \calA : \Psi^{\nu}_i(\bsfU) \geq 0, \, \forall j\in\calI(i)\}.
\end{equation}
%
%
Then $\bsfU_i\upLnp \in \calB_i \subset \calA$ (recall $\calA$ is defined in \eqref{eq:invariant_set}).
This property is stronger than the typical ``positivity-preserving'' since it includes a local minimum principle on the specific entropy.
\end{remark}

We would like to emphasize that the order of the limiting is essential. 
For example, if the high-order partial densities are negative and one tries to first limit the entropy, then the method will fail as the entropy requires the logarithm of the mixture density.
We define the interpolation between the low-order update and the high-order update as follows: 
\begin{equation}
    g^{\nu}_{ij}(\ell) \eqq \Psi_i^\nu(\bsfU_i\upLnp) + \ell \frac{\Psi_i^\nu(\bsfU_i\upLnp + \ell^j_{i,\nu - 1} \bsfP^n_{ij} ) - \Psi_i^\nu(\bsfU_i\upLnp)}{\ell^j_{i,\nu - 1} + \epsilon},
\end{equation}
for $\ell \in [0, \ell^j_{i,\nu-1}]$, $\ell_{i,0}^j \eqq 1$, and $0 < \epsilon \ll 1$ is a machine precision constant to avoid division by zero.
Note that $g^{\nu}_{ij}(0) \eqq \Psi^\nu(\bsfU_i\upLnp) \geq 0$ for all $\nu \in \intset{1}{2n_s+2}$.
The goal is to find $\ell_{i,\nu}^j\in [0,1]$ such that $g^{\nu}_{ij}(\ell_{i,\nu}^j) \geq 0$ for all $\nu \in \intset{1}{2n_s+2}$ in a sequential manner.
If $\Psi_i^\nu(\bsfU_i\upLnp + \ell^j_{i,\nu-1} \bsfP^n_{ij}) > 0$, then $\ell_{i,\nu}^j = \ell_{i,\nu-1}^j$. If this is not the case, then we find the root of $g^\nu_{ij}(\ell^j_{i,\nu}) = 0$ with a one step regula falsi approach which is given by:
\begin{equation} \label{eq:limiting_root}
    \ell^j_{i,\nu} = \min\left(\ell^j_{i,\nu}, \frac{-(\ell^j_{i,\nu-1} + \epsilon) \Psi_i^\nu(\bsfU_i\upLnp)} {\Psi_i^\nu(\bsfU_i\upLnp + \ell^j_{i,\nu - 1} \bsfP^n_{ij} ) - \Psi_i^\nu(\bsfU_i\upLnp)}\right).
\end{equation}
Under this sequential limiting, we see that $\ell^j_{i,2n_s+2} \leq \ell^j_{i,2n_s + 1} \leq \cdots \leq \ell^j_{i,1} \leq \ell^j_{i,0} = 1$ for every $i \in \calV$ and $j \in \calI(i) \setminus \{i\}$.

In order to make the limiting methodology precise, we frame the problem as the construction of a symmetric matrix, $\calL \in \text{Sym}(|\calV|, \Real)$, defined by $\calL \eqq \min(L, L\tr)$ (with the $\min$ operation being defined component-wise).
The entries of $L$ are given by:
\begin{equation}
    (L)_{ij} \eqq \begin{cases}
        \ell^j_{i,2n_s+2}, &\text{ if } j \in \calI(i) \setminus \{i\}, \\
        1, &\text{ otherwise.}
    \end{cases}
\end{equation}
Note that the symmetrization of the limiter guarantees global mass conservation~\cite[Sec.~4.2]{guermond2018second}.
The algorithm for computing the limiter for each $(i,j)$ pair is given in Algorithm~\ref{alg:limiter}
\begin{algorithm}
\caption{Compute the limiter $\ell_{ij}\upn$ for the pair $(i,j)$}
\label{alg:limiter}

    \KwIn{$\{\bsfU_i\upLnp\}$, $\{\bsfP^n_{ij}\}$}
    \KwOut{$\calL$}

    $(L)_{ij} = 1$, for all $i,j \in \calV$.
    
    \For{$i \in \calV$}{
        \For{$j \in \calI(i) \setminus\{i\}$}{
            \For{$\nu \in \intset{1}{2n_s+2}$}{
                \If{$\Psi_i^\nu(\bsfU_i\upLnp + \ell^j_{i,\nu-1} \bsfP^n_{ij}) \geq 0$}{
                    $\ell^j_{i,\nu} = \ell^j_{i,\nu-1}$. \;
                }
                \Else{
                    \textbf{Compute} $\ell^j_{i,\nu}$ from \eqref{eq:limiting_root}. \;
                }
                
            } 

            $(L)_{ij} \eqq \ell^j_{i,2n_s+2}$
        }
    }
    $\calL \eqq \min(L, L\tr)$
\end{algorithm}
We now give the main result of the paper.
\begin{theorem}[Invariant-domain preserving] \label{thm:idp_property}
     Let the limited update $\bsfU_i\upnp$ be defined by~\eqref{eq:high_order_update}  for all $i\in\calV$ combined with limiter procedure outlined in Algorithm~\ref{alg:limiter} for all $j\in\calI(i)$. Then, $\bsfU_i\upnp$ is globally mass-conservative and satisfies the local bounds~\eqref{eq:concave_functionals}, $\Psi_i^k(\bsfU\upLnp_i) \geq 0$, for all $k\in\intset{1}{2n_s + 2}$. That is, the update $\bsfU_i\upnp$ is invariant-domain preserving.
\end{theorem}

\subsection{Relaxation of bounds} \label{sec:relaxation}
It is known that one must relax the bounds for achieving optimal convergence in the $L^\infty$ norm (see:~\cite{khobalatte1994maximum} and~\cite[Sec.~4.7]{guermond2018second}). 
In this work, we directly follow~\citep[Sec.~4.7.1]{guermond2018second} for the relaxation of the partial density bounds and the mixture internal energy bound. As opposed to~\citep{guermond2018second}, we propose a different relaxation of the specific entropy 
bound $s^{\min}$: 
\begin{equation} \label{eq:entropy_relax}
    s^{\min,i}_{\text{relax}} \eqq \max\Big( c_v(\bY_i) \log \Big[(1 - r_{h,i}) \exp\Big( \frac{s^{\min, i}}{c_v(\bY_i)} \Big) \Big], s^{\min,i} - \Delta s^{\min,i} \Big),
\end{equation}
where $r_{h,i} = \big(\frac{m_i}{|D|}\big)^{1.5/d}$  and
\begin{equation}
    \Delta s^{\min,i} = \max_{j\in\calI(i) \setminus \{i\}} s\big( \tfrac12( \bsfU^n_i + \bsfU^n_j ) \big) - s^{\min,i}.
\end{equation}
We note that the physical entropy in~\citep{guermond2018second} was assumed to always be positive. However, the physical entropy of the mixture~\eqref{eq:mixture_entropy} can be negative. Thus, if the specific entropy happens to be close to zero, then typical relaxation of~\citep{guermond2018second} in the form $(1 - r_{h,i}) s^{\min,i}$ fails to provide a proper relaxation, hence the reason for the ``log-exp'' transformation in \eqref{eq:entropy_relax}.
We further note that the relaxation on the physical entropy leads to a ``weak'' enforcement of the minimum entropy principle as stated in~\citep{guermond2018second}. This is observed numerically. Without the relaxation, the minimum entropy principle is exactly enforced.

%% file: illustrations.tex
\section{Numerical illustrations}\label{sec:illustrations}
We now illustrate the proposed methodology.
In particular, we \textup{i)} verify the accuracy of the numerical method with smooth analytical solutions and an exact solution to the Riemann problem; \textup{ii)} compare with standard benchmarks in the literature; \textup{iii)} validate the proposed model by comparing with small-scale experiments.

\subsection{Preliminaries}
The numerical tests are performed with the high performance code,~\texttt{ryujin}~\citep{ryujin-2021-1, ryujin-2021-3}. The code uses continuous $\polQ_1$ finite elements on quadrangular meshes for the
spatial approximation and is built upon the \texttt{deal.II} finite element
library~\citep{dealII95}. For all tests, the time-stepping is done with a three stage, third-order Runge-Kutta method which is made to be invariant-domain preserving following the techniques introduced in~\cite{Ern_Guermond_2022}. The time step size is defined by $\tau\eqq 3 \tau_n$ where $\tau_n$ is computed during the first stage of each time
step using:
\begin{equation*}
    \dt_n \eqq \textup{CFL} \max_{i\in\calV}\frac {m_i}{2|d_{ii}\upLn|},
\end{equation*}
where $\textup{CFL}\in(0,1]$ is a user-defined parameter. For the sake of simplicity, we set $\textup{CFL} = 0.5$ for all tests. All units are assumed to be SI units unless otherwise stated.

\subsection{Verification}
We now verify the accuracy of the numerical method. To quantify the error, we introduce the following consolidated error indicator at time $t$ which measures the $L^q$-norm for all components of the system:
\begin{equation}\label{eq:error}
    \delta^q(t)\eqq\sum_{k=1}^{m}\frac{\|\bu_{k,h}(t) - \bu_k(t)\|_q}{\|\bu_k(t)\|_q}.
\end{equation}
Here, $\bu_k(t)$ is the exact state for the $k$-th component of the solution and $\bu_{h,k}(t)$ is the spatial approximation of the $k$-th component.

\subsubsection{1D two-species smooth traveling wave}
We consider a two-species extension of the one-dimensional test proposed in~\citep{guermond2018second}. The test consists of a two traveling partial density waves with constant mixture pressure and mixture velocity.
Let $\rho_0 = \SI{1}{kg\,m^{-3}}$ be the ambient mixture density.
The partial density profiles are defined by: %
\begin{subequations}
    \begin{equation}
        \rho(x, t) = \begin{cases}
            \rho_0 + 2^6(x_1 - x_0)^{-6}(x - v_0 t - x_0)^3(x_1 - x + v_0 t)^3, & \text{if }x_0\leq x - v_0 t\leq x_1, \\
            \rho_0,                                                             & \text{otherwise,}
        \end{cases}
    \end{equation}
    \begin{equation}
        (\alpha_1\rho_1)_0(x, t) = \frac34\times\rho(x, t), \qquad (\alpha_2\rho_2)_0(x, t) = \frac14\times\rho(x, t),
    \end{equation}
\end{subequations}
where $x_0 = \mynum{0.1}{m}$ and $x_0 = \mynum{0.3}{m}$. The mixture pressure and velocity are set to $p(\bx, t)=\mynum{1}{Pa}$ and $v(x, t) = \mynum{1}{m\,s^{-1}}$, respectively. Each species is characterized by the equation of state parameters $\gamma_1 = \frac{1005}{718}$ and $\gamma_2 = \frac{4041.4}{2420}$.
The computational domain is defined by $D = (0, \SI{1}{m})$ with Dirichlet boundary conditions. The tests are performed on a sequence of uniform meshes. The final time is set to $t_f = \mynum{0.6}{s}$. We report the consolidated error $\delta_q(t_f)$ for $q=\{1,2,\infty\}$ and respective convergence rates in Table~\ref{tab:smooth_wave}. We observe optimal convergence rates.
\begin{table}[htbp!]
    \centering
    \begin{tabular}[b]{rrlrlrl}
        \toprule
        $I$   & $\delta^1(t_f)$             &      & $\delta^2(t_f)$             &      & $\delta^{\infty}(t_f)$             \\
        101   & \num{0.01737930751465193}   &      & \num{0.0451890465480829}    &      & \num{0.1461546759394965}    &      \\
        201   & \num{0.003145152959788272}  & 2.47 & \num{0.01021421538305274}   & 2.15 & \num{0.04209201481416872}   & 1.8  \\
        401   & \num{0.000277534116248811}  & 3.5  & \num{0.0008982376509427262} & 3.51 & \num{0.004929716239465058}  & 3.09 \\
        801   & \num{1.789543153519735e-05} & 3.96 & \num{4.747959579829728e-05} & 4.24 & \num{0.0002101900758007725} & 4.55 \\
        1601  & \num{1.877863784684281e-06} & 3.25 & \num{6.038712429264608e-06} & 2.97 & \num{3.766542496180426e-05} & 2.48 \\
        3201  & \num{2.254633347875247e-07} & 3.06 & \num{8.542662344561827e-07} & 2.82 & \num{6.966637539098599e-06} & 2.43 \\
        6401  & \num{2.69448082167417e-08}  & 3.06 & \num{1.211241599289386e-07} & 2.82 & \num{1.278280206040598e-06} & 2.45 \\
        12801 & \num{3.205858942018182e-09} & 3.07 & \num{1.721189997409589e-08} & 2.82 & \num{2.368742379276548e-07} & 2.43 \\
        25601 & \num{3.859142205735488e-10} & 3.05 & \num{2.451218217871037e-09} & 2.81 & \num{4.383276644141776e-08} & 2.43 \\
        \bottomrule
    \end{tabular}
    \caption{Errors and convergence rates for 1D two-species smooth traveling wave problem.}\label{tab:smooth_wave}
\end{table}
\subsubsection{Riemann problems}
We now verify the proposed method by comparing with exact solutions to the Riemann problem which is provided in Section~\ref{app:RP}.
In this paper, we use the Riemann data given in~\citep[Tab.~2]{renac2021entropy} for tests ``RP1'' and ``RP2''. We recall the respective details in Table~\ref{tab:RP_data}. The Riemann data is given for the variable $\bw(x,t)\eqq(Y_1, \rho, v, p)^{\mathsf{T}}$. We set $Y_2 = 1 - Y_1$. The conserved partial densities are set by $\alpha_1\rho_1 = Y_1\rho$ and $\alpha_2\rho_2 = Y_2\rho$.
\begin{table}[htbp!]
    \centering
    \begin{tabular}{ccclcc}
                   & $\bw_L$                                 & $\bw_R$                                 & $t_f$             & $\gamma_1$         & $\gamma_2$            \\
        \toprule
        \text{RP1} & $(0.5, 1, 0, 1)^{\mathsf{T}}$           & $(0.5, 0.125, 0, 0.1)^{\mathsf{T}}$     & $\mynum{0.2}{s}$  & $\frac{1.5}{1.0}$  & $\frac{1.3}{1.0}$     \\
        \midrule
        \text{RP2} & $(1, 1.602, 0, \num{1e6})^{\mathsf{T}}$ & $(0, 1.122, 0, \num{1e5})^{\mathsf{T}}$ & $\mynum{3e-4}{s}$ & $\frac{5.2}{3.12}$ & $\frac{1.402}{0.743}$ \\
        \bottomrule
    \end{tabular}
    \caption{Initial conditions and problem setup for the 1D Riemann problems.}\label{tab:RP_data}
\end{table}
For each test, the computational domain is defined by $D = (0, \SI{1}{m})$ with Dirichlet boundary conditions. The diaphragm is set to $x_0 = \mynum{0.5}{m}$.
The convergence tests are performed on a sequence of uniform meshes. We observe an asymptotic rate of 1 which is expected in the $L^1$-norm.
We show the output for both Riemann problems with various refinement levels in \cref{fig:rp}.
\begin{table}[htbp!]
    \centering
    \begin{subtable}[t]{0.45\textwidth}
        \centering
        \begin{tabular}[b]{rrlrlrl}
            \toprule
            $I$   & $\delta^1(t_f)$            &      \\
            101   & \num{0.06909400157298831}  &      \\
            201   & \num{0.03586304487271475}  & 0.95 \\
            401   & \num{0.01971692340664219}  & 0.86 \\
            801   & \num{0.01167814873120082}  & 0.76 \\
            1601  & \num{0.007606564787683386} & 0.62 \\
            3201  & \num{0.004424058777355496} & 0.78 \\
            6401  & \num{0.002472830353424289} & 0.84 \\
            12801 & \num{0.001269588897949639} & 0.96 \\
            \bottomrule
        \end{tabular}
        \caption{RP1}
        \label{tab:RP_1}
    \end{subtable}
    \begin{subtable}[t]{0.45\textwidth}
        \centering
        \begin{tabular}[b]{rrlrlrl}
            \toprule
            $I$   & $\delta^1(t_f)$            &      \\
            101   & \num{0.2897800802921233}   &      \\
            201   & \num{0.1808748820047454}   & 0.68 \\
            401   & \num{0.08091608336112947}  & 1.16 \\
            801   & \num{0.05386850958187564}  & 0.59 \\
            1601  & \num{0.03477500532249028}  & 0.63 \\
            3201  & \num{0.02086892952882959}  & 0.74 \\
            6401  & \num{0.01114776898925952}  & 0.9  \\
            12801 & \num{0.005543618410043464} & 1.01 \\
            \bottomrule
        \end{tabular}
        \caption{RP2}
        \label{tab:RP_2}
    \end{subtable}
    \caption{Errors and convergence rates for 1D Riemann problems.}
\end{table}

\begin{figure}[htbp!]
    \centering
    \subfloat[RP1]{\adjustbox{width=0.49\linewidth, valign=b}{\input{FIGS/rp1_d}}
        {\adjustbox{width=0.49\linewidth, valign=b}{\input{FIGS/rp1_p}}}
    }
    \newline
    \subfloat[RP2]{\adjustbox{width=0.49\linewidth, valign=b}{\input{FIGS/rp2_d}}{\adjustbox{width=0.49\linewidth, valign=b}{\input{FIGS/rp2_p}}}}
    \caption{Mixture density (left) and pressure (right) at $t_f$ for the 1D Riemann problems computed using varying mesh resolutions. \label{fig:rp}
    }
\end{figure}
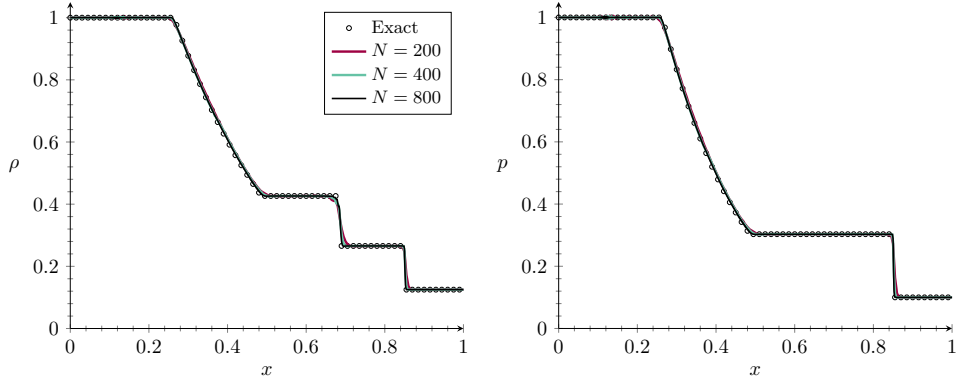
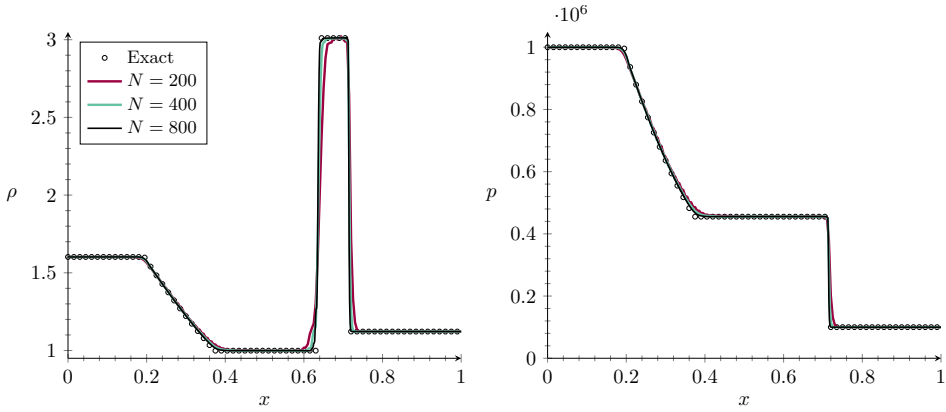

\subsection{Benchmarks}
We now benchmark the efficacy of the proposed scheme in terms of canonical flow problems in the literature as well as novel problem configurations.

\subsubsection{2D -- Shock-bubble interaction}
We first consider a two-dimensional shock-bubble interaction. Although there are many variations of this problem in the literature (\eg \citep{niederhaus2008computational, quirk1996dynamics}), we choose to simulate the physical experiment described in~\cite{layes2009experimental}.
The experiment consists of a shock wave traveling at Mach 1.43 in air ($\gamma_1 = \frac{1005}{718}$) colliding with a krypton bubble ($\gamma_2 = \frac{248}{149}$).
We note that the physical setup utilized a ``straw'' to fill the krypton in a thin soap bubble to prevent the gas from diffusing into the air.
However, we note that the encapsulating soap bubble cannot be described by the current model and the ``straw'' cannot be properly modeled in two dimensions.
Thus, we consider the air and krypton as the only species in the numerical simulation.
Let $\rho_\text{shock} = 2.025655508041382
    ~\SI{}{ kg~m^{-3}},\, v_\text{shock} = 212.66552734375~\SI{}{m~s^{-1}},\, p_\text{shock} = 224835~\SI{}{ Pa}$.
Then, the initial state is given as follows:
\begin{equation} \label{eq:shock_bubble_initial}
    \bw(\bx, t)
    \eqq
    \begin{cases}
        (1, \rho_\text{shock}, (v_\text{shock}, 0), p_\text{shock}), & \quad \text{ if } x_1 < \mynum{0.03}{m}                              \\
        (0, 3.408, \mathbf{0}, 101325),                              & \quad \text{ if } \Vert \bx - \bx_b\Vert_\ell \leq \mynum{0.022}{m}, \\
        (1, 1.163,  \mathbf{0}, 101325),                             & \quad \text{ otherwise.}
    \end{cases}
\end{equation}
where $\bx_b = (\mynum{0.052}{m}, \mynum{0.04}{m})$ denotes the center of the bubble.

The computational domain is set to $D = (\mynum{-0.12}{m}, \mynum{0.88}{m})\times(0, \mynum{0.08}{m})$. To ensure mesh convergence (or at least close to), we run the simulation with 20,496,281 $\polQ_1$ DOFs which corresponds to 16,000 elements in the $x$-direction and 1,280 elements in the $y$-direction.
The final time is set to $t_{\text{f}} = \mynum{1230}{\mu s}$.
The numerical Schlieren for the partial density $\alpha_1 \rho_1$ plot is compared with the results in \cite[Fig.~5]{layes2009experimental}.
In \cref{fig:schlieren_shock_bubble}, one sees that the vorticity in the physical experiment is not as noticeable compared to the numerical Schlieren.
This difference has been discussed in \citet[Sec.~4.~B]{layes2007quantitative} and \citet{giordano2006richtmyer}.
\begin{figure}[htbp!]
    \centering
    \includegraphics[clip,trim={2 0 0 0}, width=0.95\linewidth]{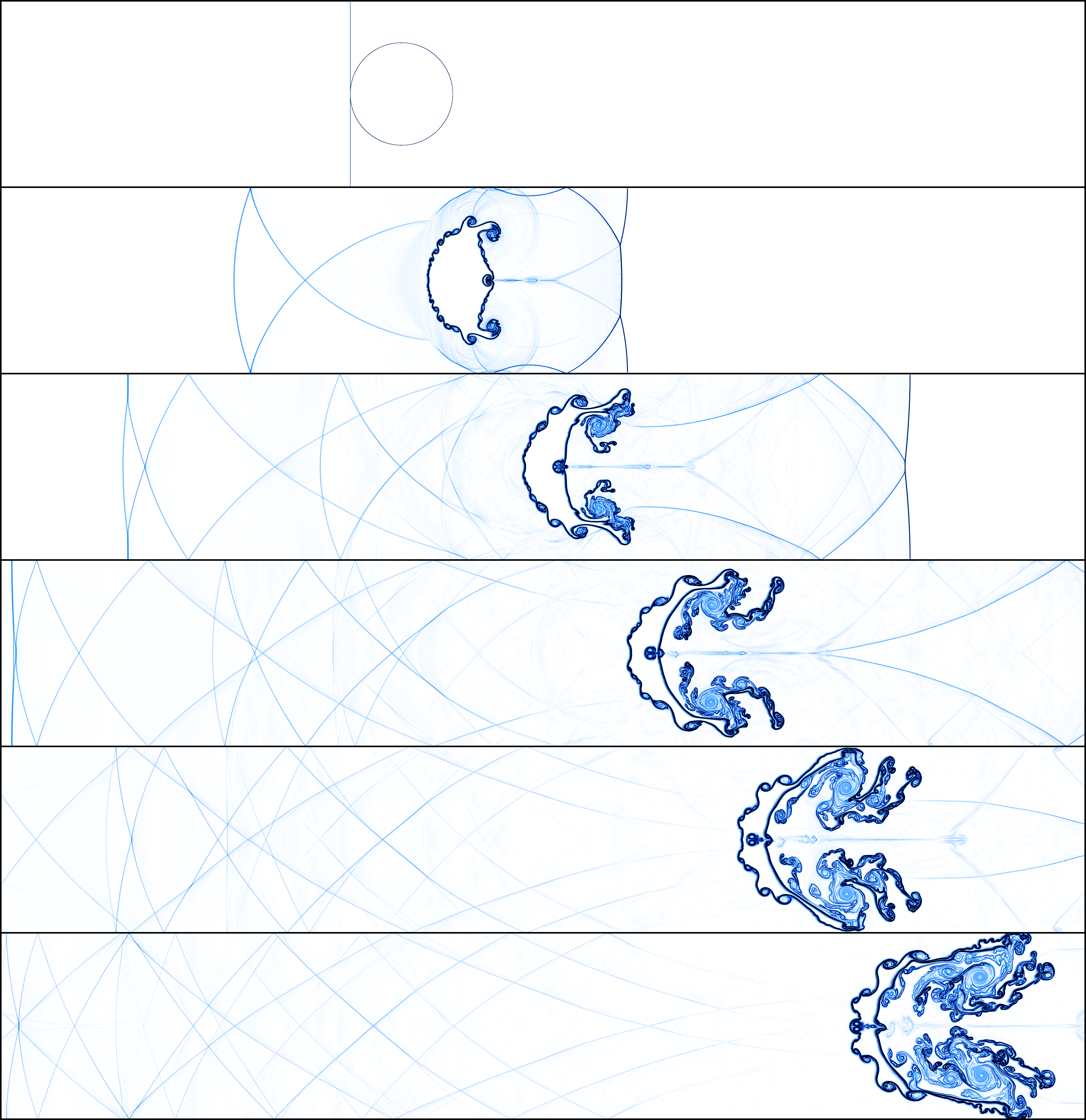}
    \caption{2D Shock-bubble -- Numerical schlieren output (with respect to the air partial density) at $t = \{0, 246, 492, 738, 984, 1230\}\SI{}{\mu s}$.
    }
    \label{fig:schlieren_shock_bubble}
\end{figure}
\begin{figure}
    \centering
    \includegraphics[width=0.95\linewidth]{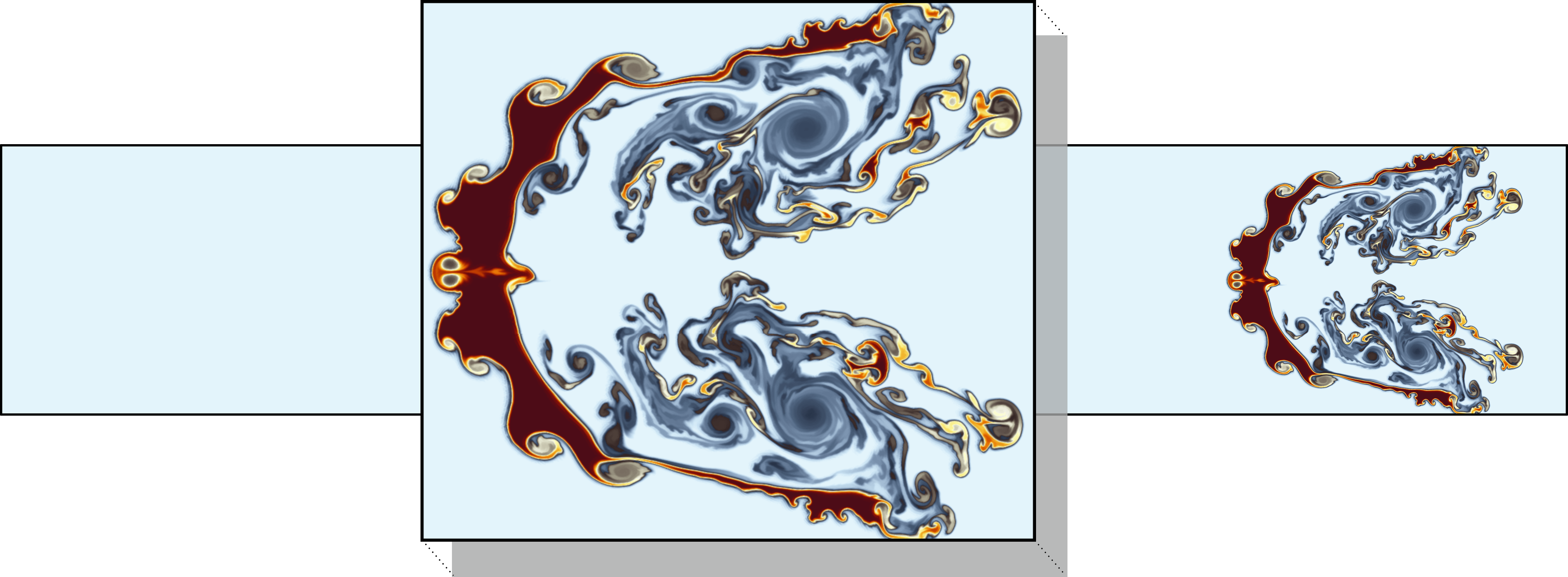}
    \caption{2D Shock-bubble -- A zoomed in snapshot of the mass fraction for krypton at $t = \mynum{1230}{\mu s}$.}
    \label{fig:zoomed_in_shock_bubble_mass_frac}
\end{figure}
Overall, we resolve the typical flow features of standard shock-bubble type problems and compare well with the simulations presented in~\cite[Fig.~7]{layes2007quantitative}.

\subsection{2D -- Simplified inertial confinement fusion configuration}
We now perform a simulation of a multi-species implosion problem akin to inertial confinement fusion (ICF)-type configurations inspired by~\cite[Sec.~6.6]{bello2020matrix}. A demonstration of physical experiments with similar setups can be found in~\citet{li2020convergent}. The problem consists of a circular Mach 5 shock wave converging towards a species interface inducing Richtmyer--Meshkov instabilities in the flow. These instabilities are seeded by perturbations in the interface, which drive the flow into a chaotic mixing state with distinct vortical structures. We simulate a shock wave moving through ambient air ($\gamma_1 = \frac{1.008}{0.72}$) with an internal helium region ($\gamma_2 = \frac{5.1932}{3.115}$). Let $\rho_\text{shock} = 5.002322673797607
    ~\SI{}{ kg~m^{-3}},\, v_\text{shock} = 1.4966877698898315~\SI{}{m~s^{-1}},\, p_\text{shock} = 2.8997678756713867~\SI{}{ Pa}$.
The set up is as follows:
\begin{equation} \label{eq:icf_initial}
    \bw(\bx, t)
    \eqq
    \begin{cases}
        (1,\rho_{\text{shock}}, -v_\text{shock}\frac{\bx}{\|\bx\|}, p_\text{shock}), & \quad \text{ if } 1.1 < \|\bx\|,                         \\
        (1, 1, \bm{0}, 0.1),                                                         & \quad \text{ if } 1 + 0.02\cos(8\theta) < \|\bx\| < 1.1, \\
        (0, 0.05, \bm{0}, 0.1),                                                      & \quad \text{ otherwise,}
    \end{cases}
\end{equation}
where $\theta = \arctan(x_2 / x_1)$. See \cref{fig:icf-initial} for the visual representation of the initial conditions.

The computational domain is the disk characterized by $R = \SI{1.2}{m}$ centered at $(0,0)$. The mesh is composed of 12,582,912 elements with 12,587,009 $\polQ_1$ degrees of freedom. We enforce Dirichlet conditions on the boundary. The final time is set to $t_f=\mynum{0.5}{s}$. The contours of density (left) and mass fraction $Y_1$ (right) are shown in \cref{fig:icf} at various times $t = \{\mynum{0.2}{s}, \mynum{0.4}{s}, \mynum{0.6}{s}\}$.
The effects of the interface perturbation can be seen initially in the deformation of the shock structure and contact line. These Richtmyer--Meshkov instabilities formed into several distinct quasi-radially symmetry vortical structures stemming from the peaks of the interface perturbations. The interaction of the shocks and contact discontinuities with these vortical structures then forced the flow into a more chaotic mixing state. These small-scale flow features and flow discontinuities were well-resolved by the proposed numerical approach.

\begin{figure}
    \centering
    \subfloat[$\rho(\bu)$]{\adjustbox{width=0.49\linewidth, valign=b}{\includegraphics{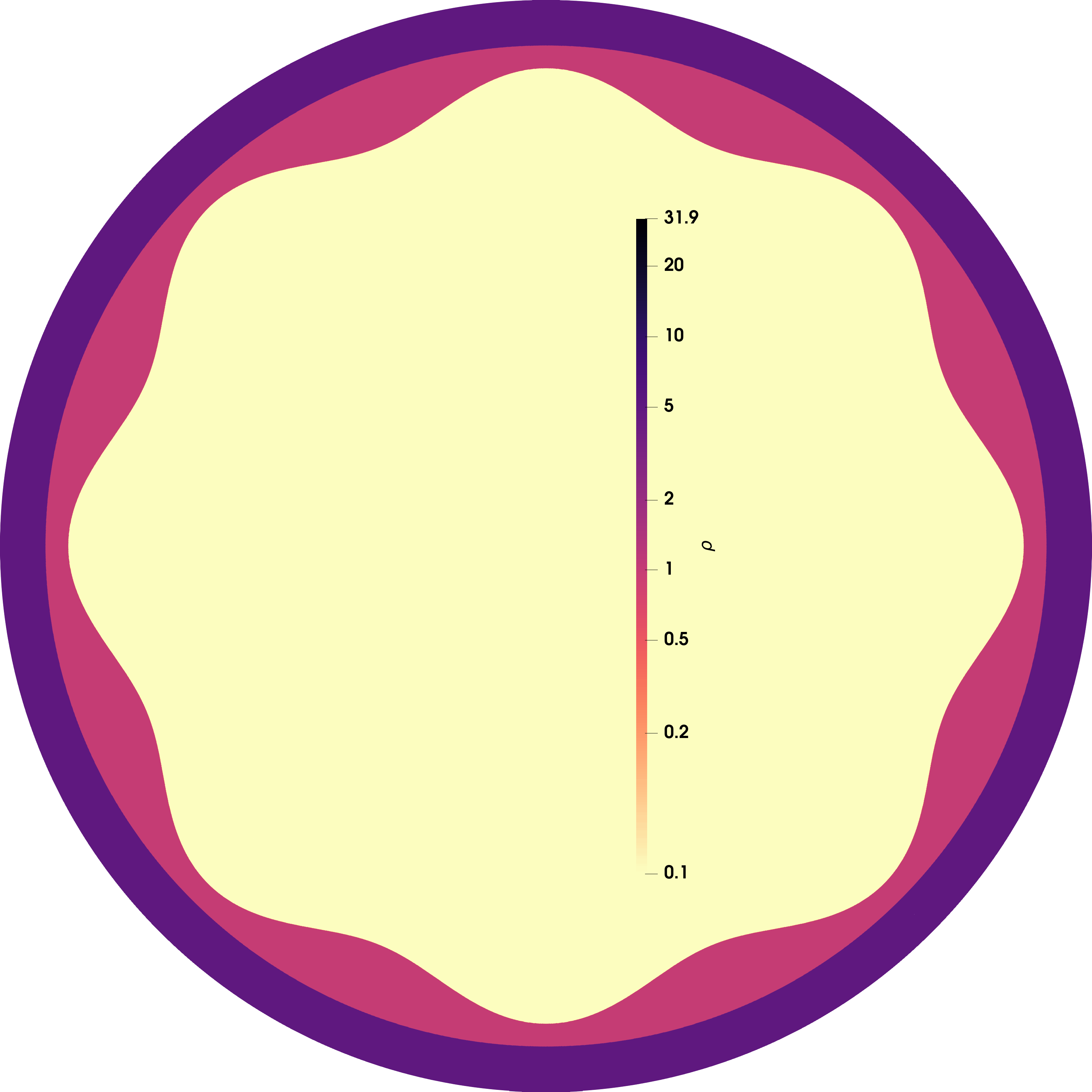}}}
    \hfill
    \subfloat[$Y_1$]{\adjustbox{width=0.49\linewidth, valign=b}{\includegraphics{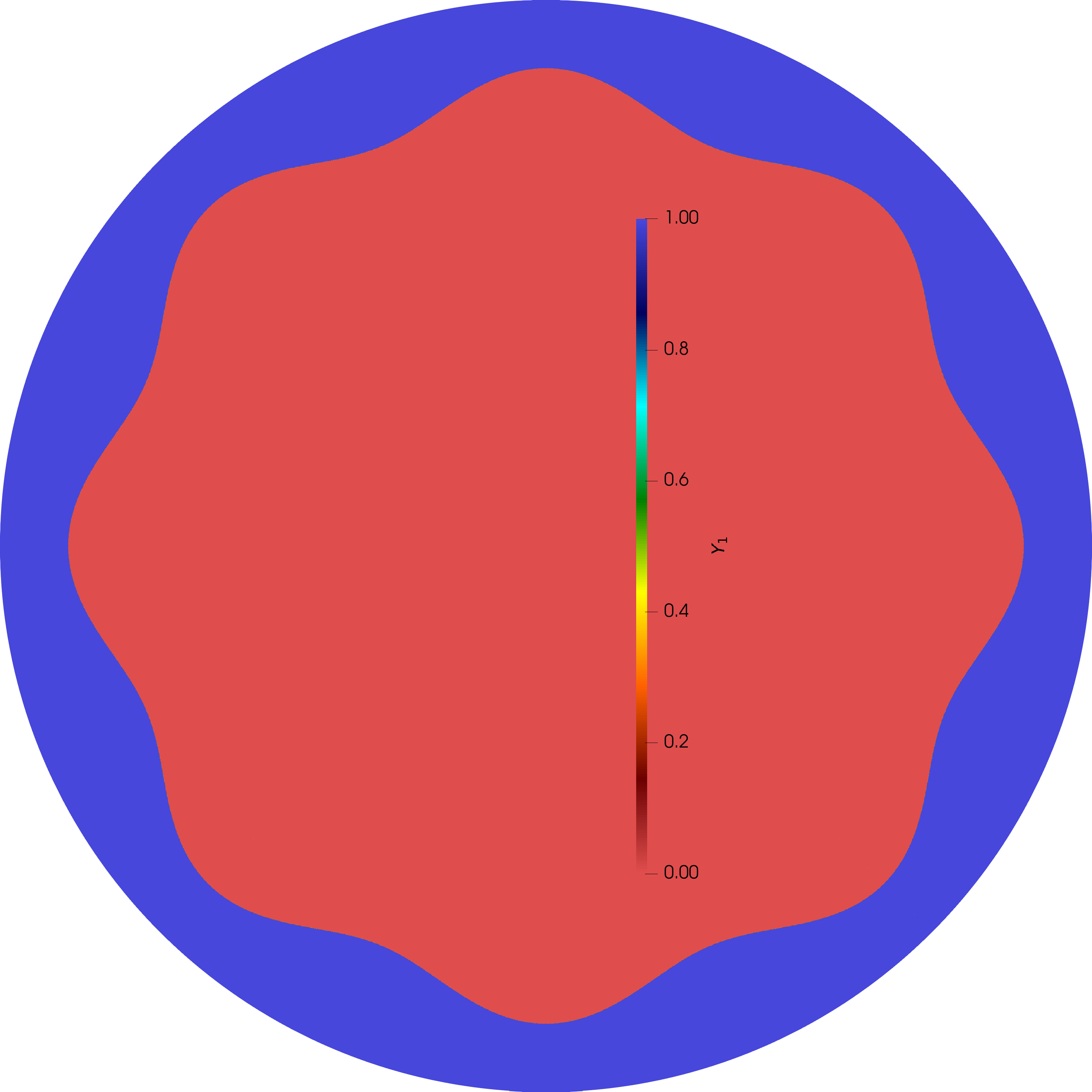}}}
    \caption{2D ICF-like problem -- Initial set up.}
    \label{fig:icf-initial}
\end{figure}

\begin{figure}
    \centering
    \subfloat[$t = 0.2s$]{\adjustbox{width=0.49\linewidth, valign=b}{\includegraphics{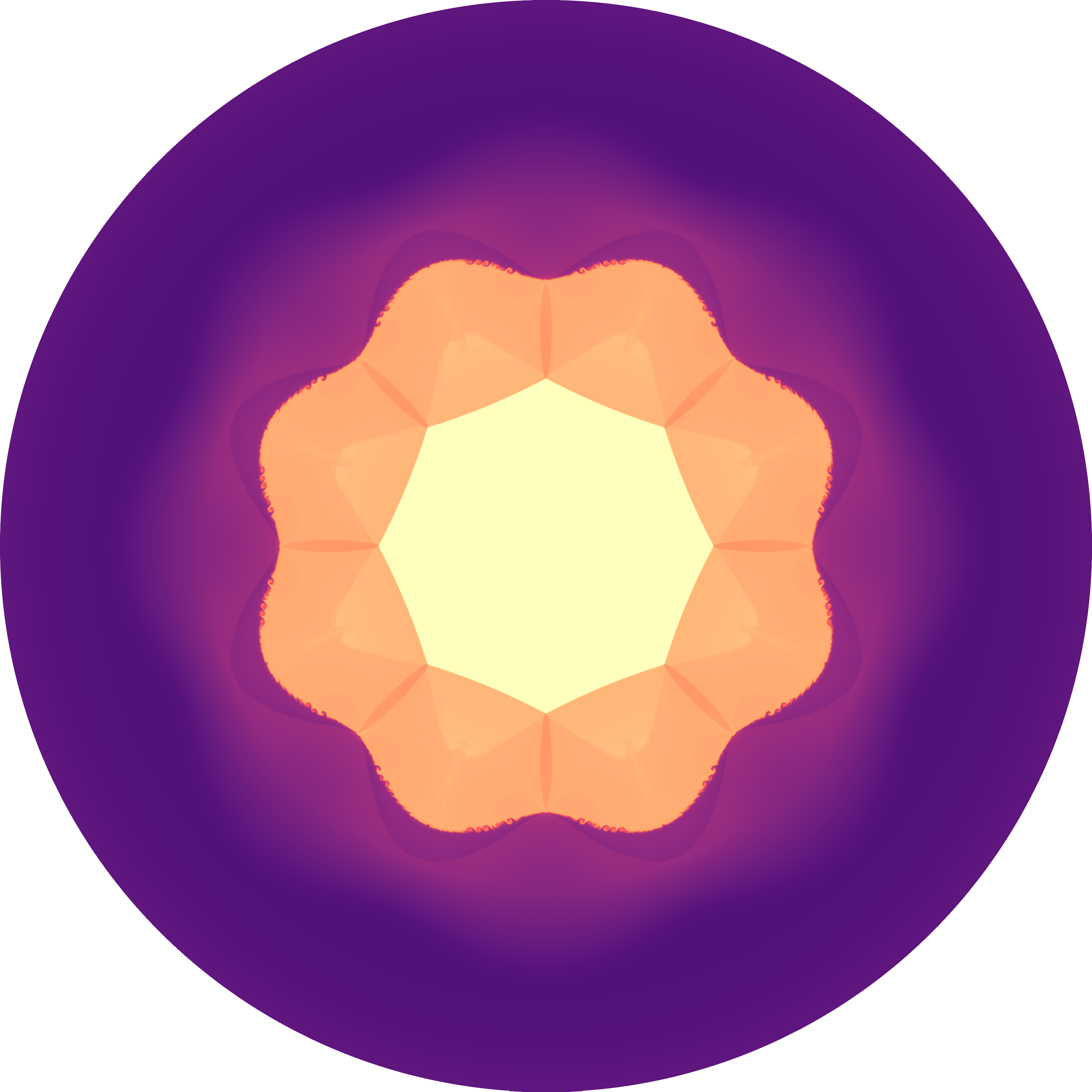}}}
    \hfill
    \subfloat[$t = 0.2s$]{\adjustbox{width=0.49\linewidth, valign=b}{\includegraphics{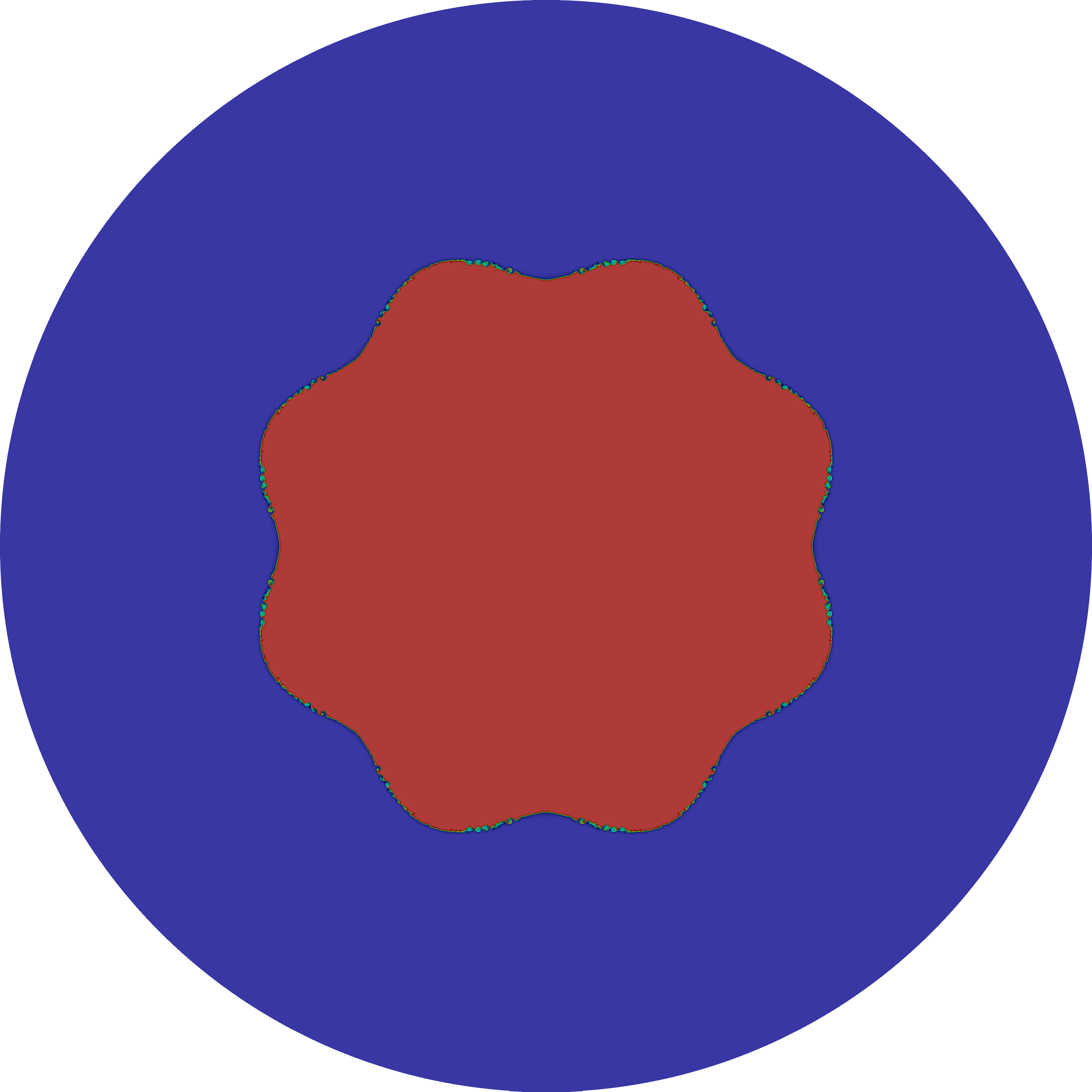}}}
    \newline
    \subfloat[$t = 0.4s$]{\adjustbox{width=0.49\linewidth, valign=b}{\includegraphics{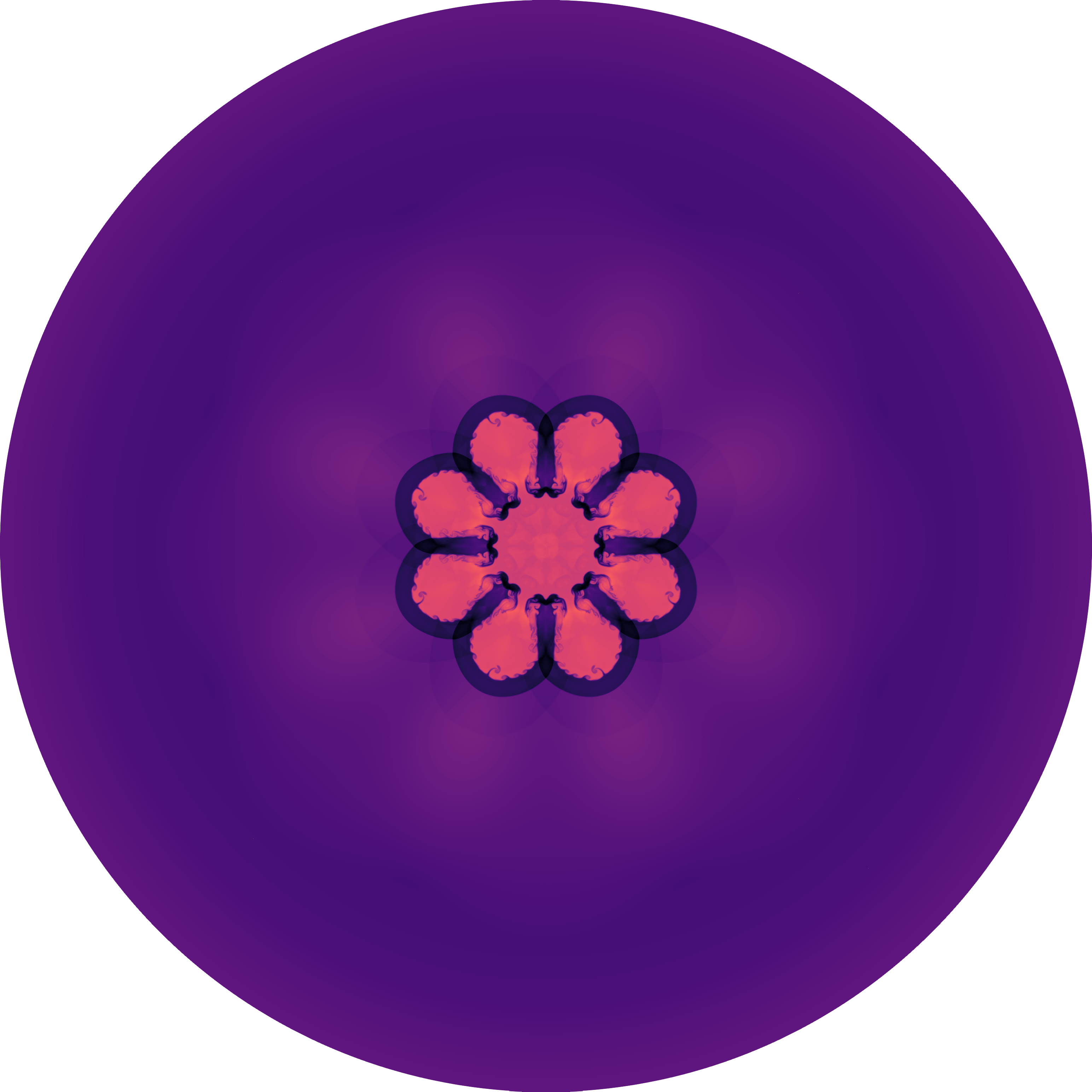}}}
    \hfill
    \subfloat[$t = 0.4s$]{\adjustbox{width=0.49\linewidth, valign=b}{\includegraphics{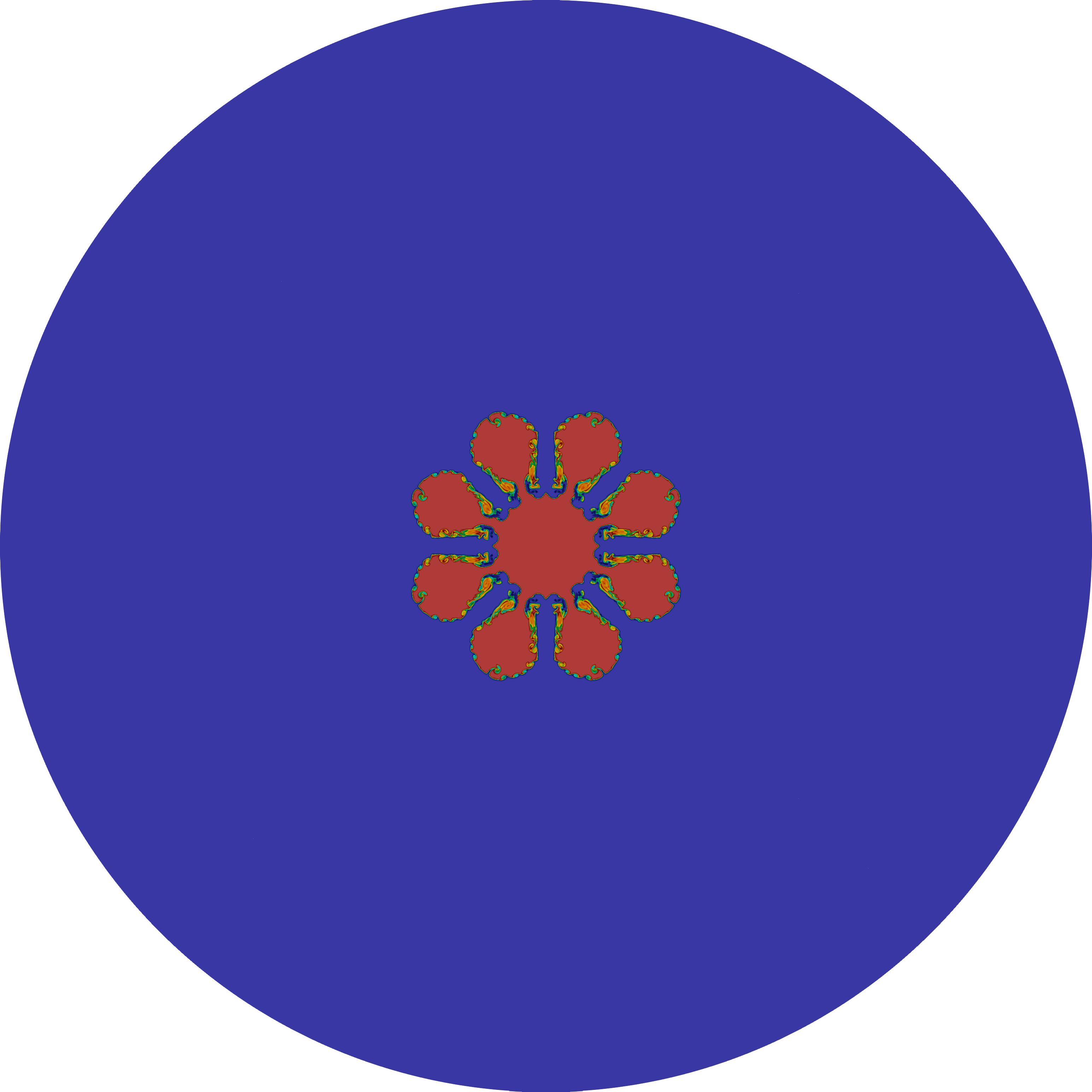}}}
    \newline
    \subfloat[$t = 0.6s$]{\adjustbox{width=0.49\linewidth, valign=b}{\includegraphics{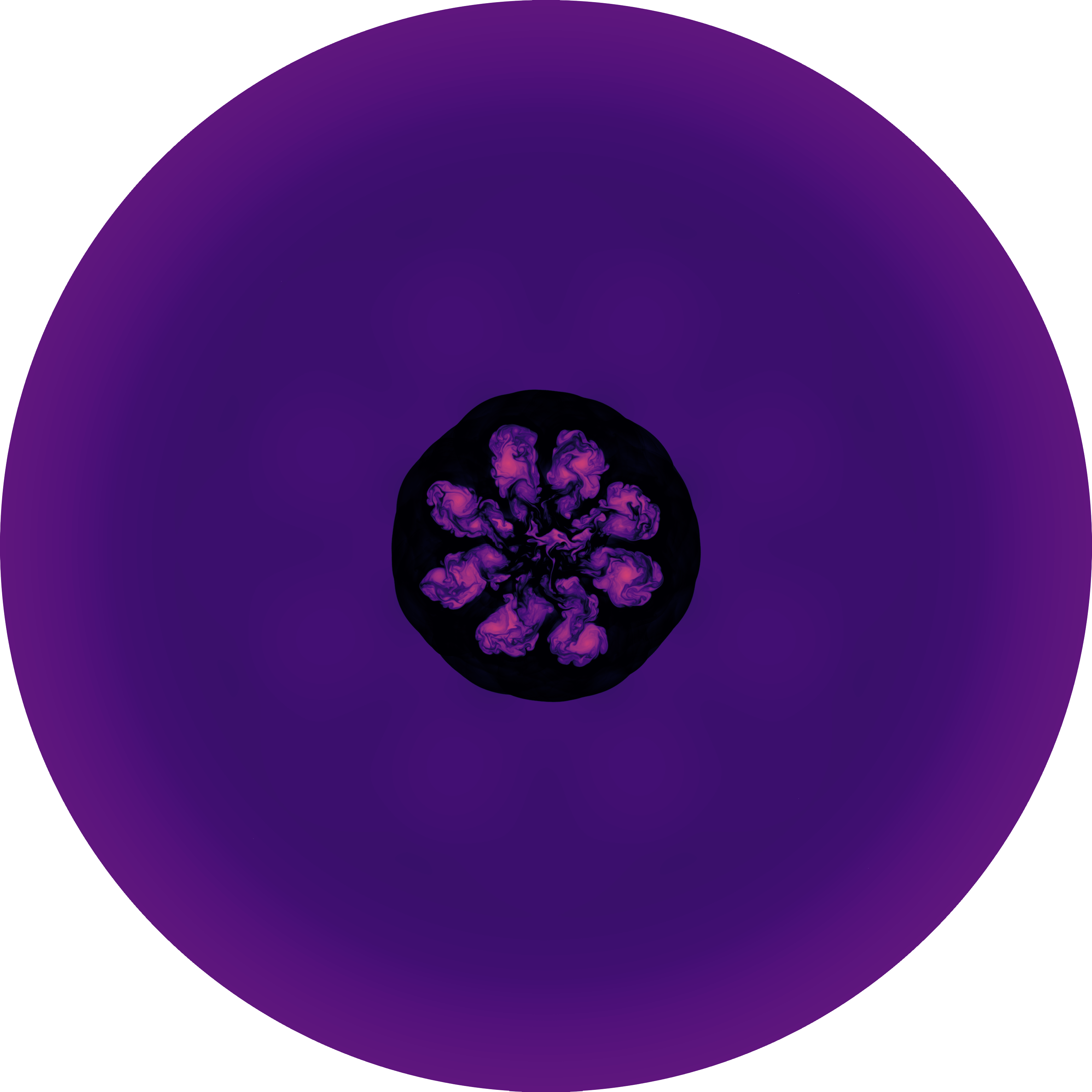}}}
    \hfill
    \subfloat[$t = 0.6s$]{\adjustbox{width=0.49\linewidth, valign=b}{\includegraphics{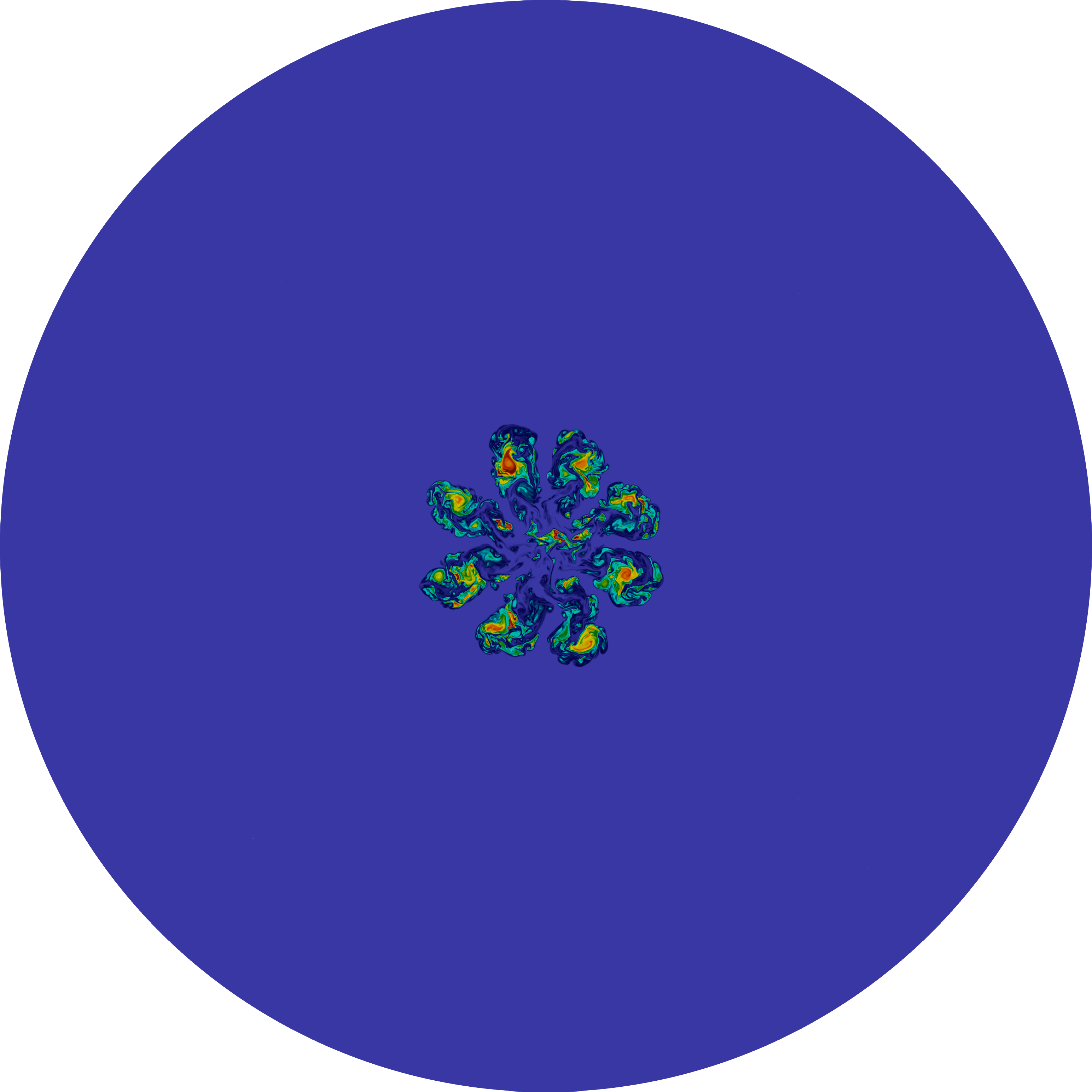}}}
    \caption{2D ICF-like problem -- Contours of density in logarithmic scale (left) and species mass fraction $Y_1$  (right) simulated using 12,587,009 $\polQ^1$ degrees of freedom at varying time intervals. Legend identical to \cref{fig:icf-initial}. }
    \label{fig:icf}
\end{figure}

\subsection{3D -- Axisymmetric triple point shock problem}
We now illustrate the method with a three-dimensional triple point problem which was originally introduced in~\citet[Sec.~8.3]{galera2010two}.
This problem is often used to demonstrate material interface tracking in Lagrangian hydrodynamics as the problem naturally induces vorticity.
Instead of the usual set up, we consder a modification which can be thought of as an ``Eulerian'' extension of the axisymmetric configuration introduced in~\citet[Sec.~4.4]{dobrev2013high}.
Here, the high-pressure ``left'' state is unmodified, the usual high-pressure high-density ``bottom'' state is now rotated about the $x_1$-axis creating a cylinder, and the usual low-pressure low-density ``top'' state acts as the ambient state outside the cylinder.
For clarity, we illustrate this in \cref{fig:3d_triple_point_initial}.

The set up is as follows. We set the parameters for each species by $\gamma_1=\frac{1.4}{1.0}$ and $\gamma_2 =\frac{1.5}{1.0}$ and the initial set up is given by:
\begin{equation}
    \bw(\bx,0) = \begin{cases}
        (0, 1, \mathbf{0}, 1),       & \text{ if } \bx \in I,   \\
        (1, 1, \mathbf{0}, 0.1),     & \text{ if } \bx \in II,  \\
        (1, 0.125, \mathbf{0}, 0.1), & \text{ if } \bx \in III,
    \end{cases}
\end{equation}
where the subregions are defined by $I \eqq \{\bx \in D : x_1 < 1\}$, $II \eqq \{\bx \in D : x_1 \geq 1, \, \sqrt{(x_2 - 1.5)^2 + x_3^2} \leq 0.5\}$, and $III \eqq D \setminus (I \cup II)$.
\begin{figure}[ht!]
    \centering
    \includegraphics[width=0.5\linewidth]{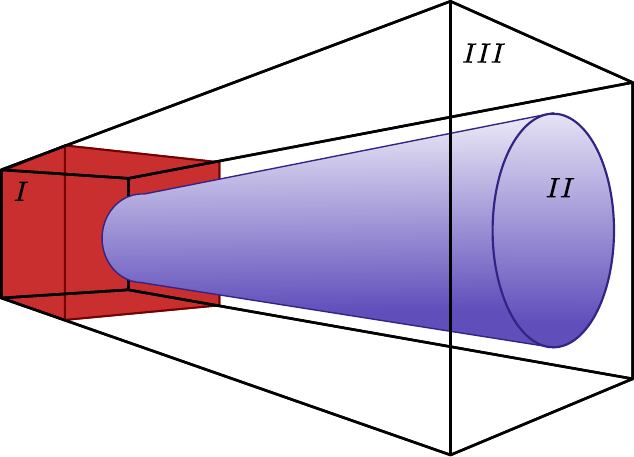}
    \caption{3D triple point problem -- Initial setup.}
    \label{fig:3d_triple_point_initial}
\end{figure}

The computational domain is $D \eqq (0,\SI{7}{m}) \times (0,\SI{3}{m})\times(0, \SI{3}{m})$ with slip boundary conditions on all boundaries. For the sake of spatial resolution demonstration, we run the simulation with two computational meshes. The first mesh, henceforth called the ``coarse mesh'', is composed of 4,114,121 $\polQ_1$ DOFs corresponding to 280 elements in the $x_1$-direction, 120 elements in the $x_2$-direction and 120 elements in the $x_3$-direction.
The second mesh, henceforth called the ``fine mesh'', is composed of of 259,355,681 $\polQ_1$ DOFs corresponding to 1120 elements in the $x_1$-direction, 480 elements in the $x_2$-direction and 480 elements in the $x_3$-direction.
The simulation is run to the final time, $t_{\text{f}} = \SI{3}{s}$.
We give the time snapshots for $t = \{\mynum{1}{s}, \mynum{3}{s}\}$ in~\cref{fig:3d_triple_point_final} for each mesh.
The representation in the figures is as follows. On the $\{x_2 = 1.5\}$ and $\{x_3 = 1.5\}$ planes, we plot mixture density in a logarithmic scale. We then plot the $\{0.5,0.6,0.7,0.8,0.9\}$ iso-volume contours of the species mass fraction $Y_1$ in a solid color. These contours are further cut along the $\{x_2 = 0.8\}$ plane for visualization purposes.
We see that the typical flow features for this problem are well resolved along the shown planes.
\begin{figure}[ht!]
    \centering
    \includegraphics[width=0.496\linewidth]{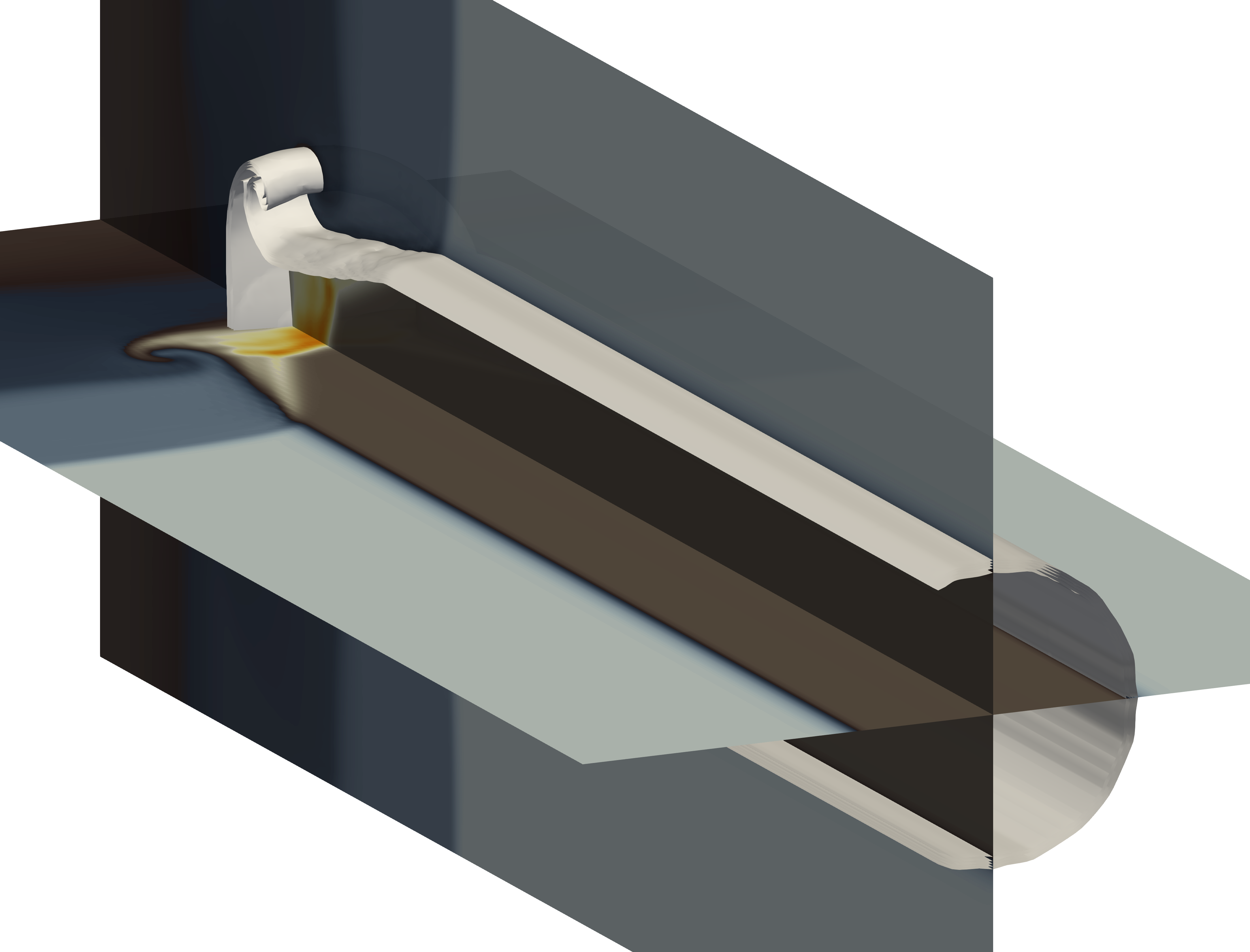}
    \includegraphics[width=0.496\linewidth]{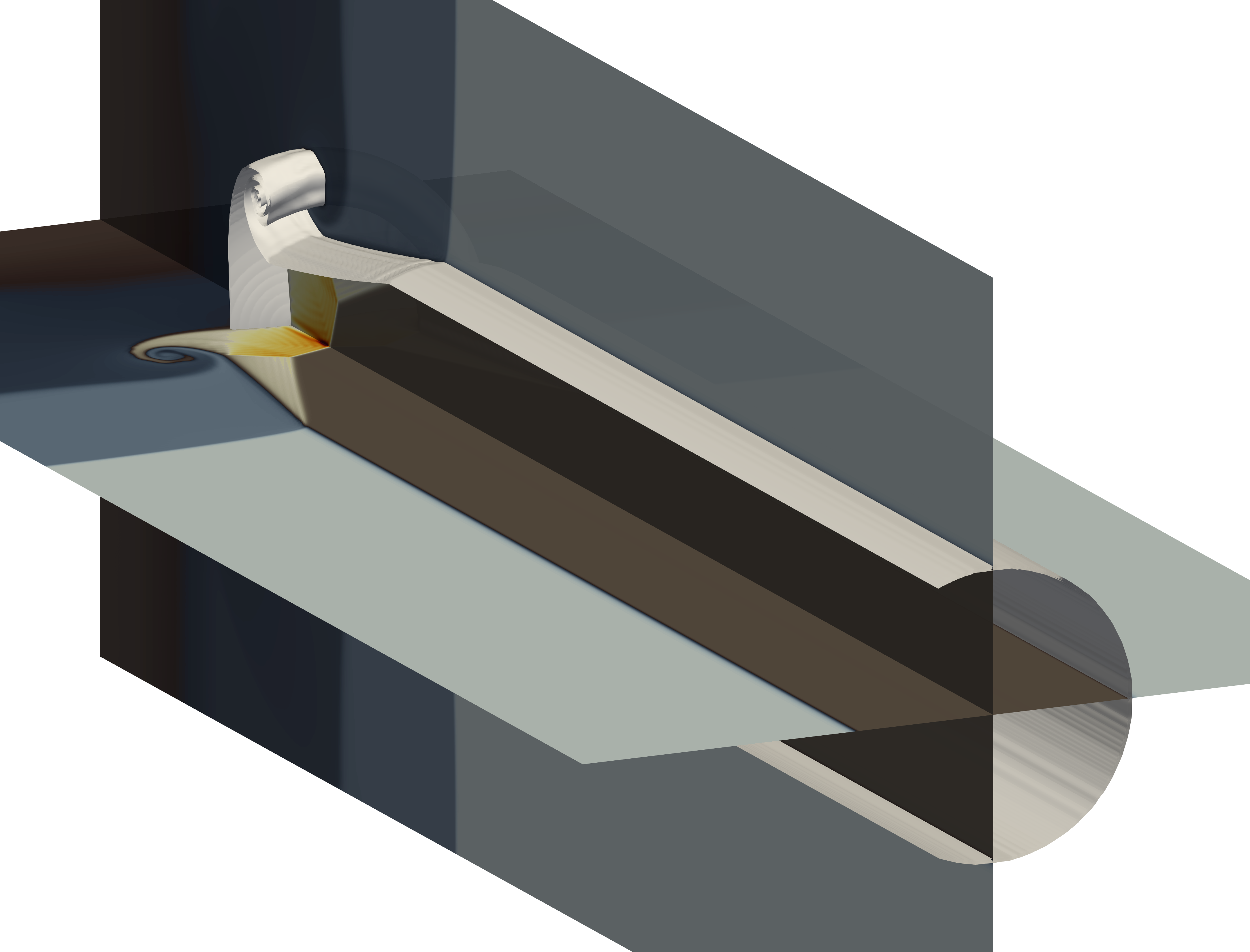}\par
    \includegraphics[width=0.496\linewidth]{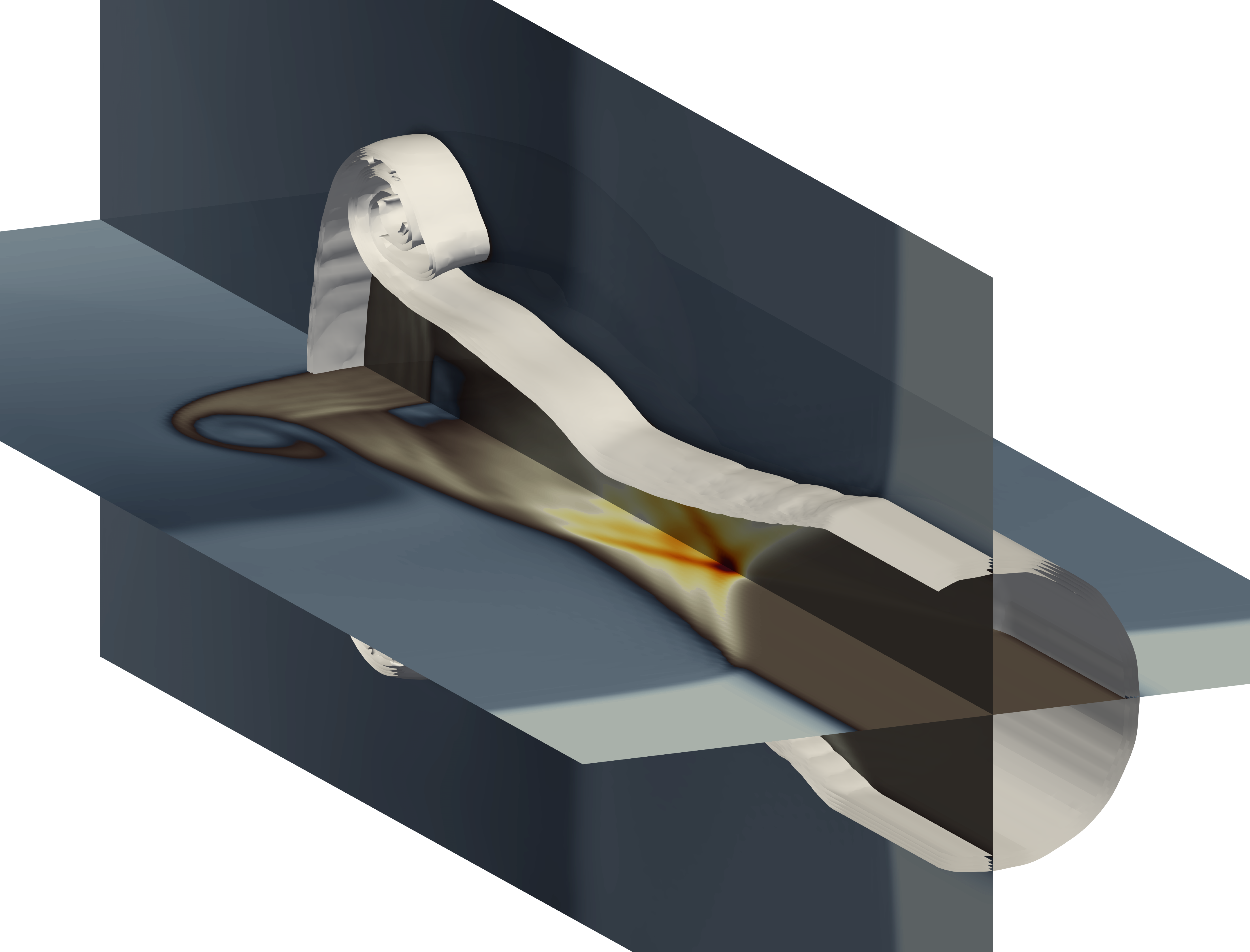}
    \includegraphics[width=0.496\linewidth]{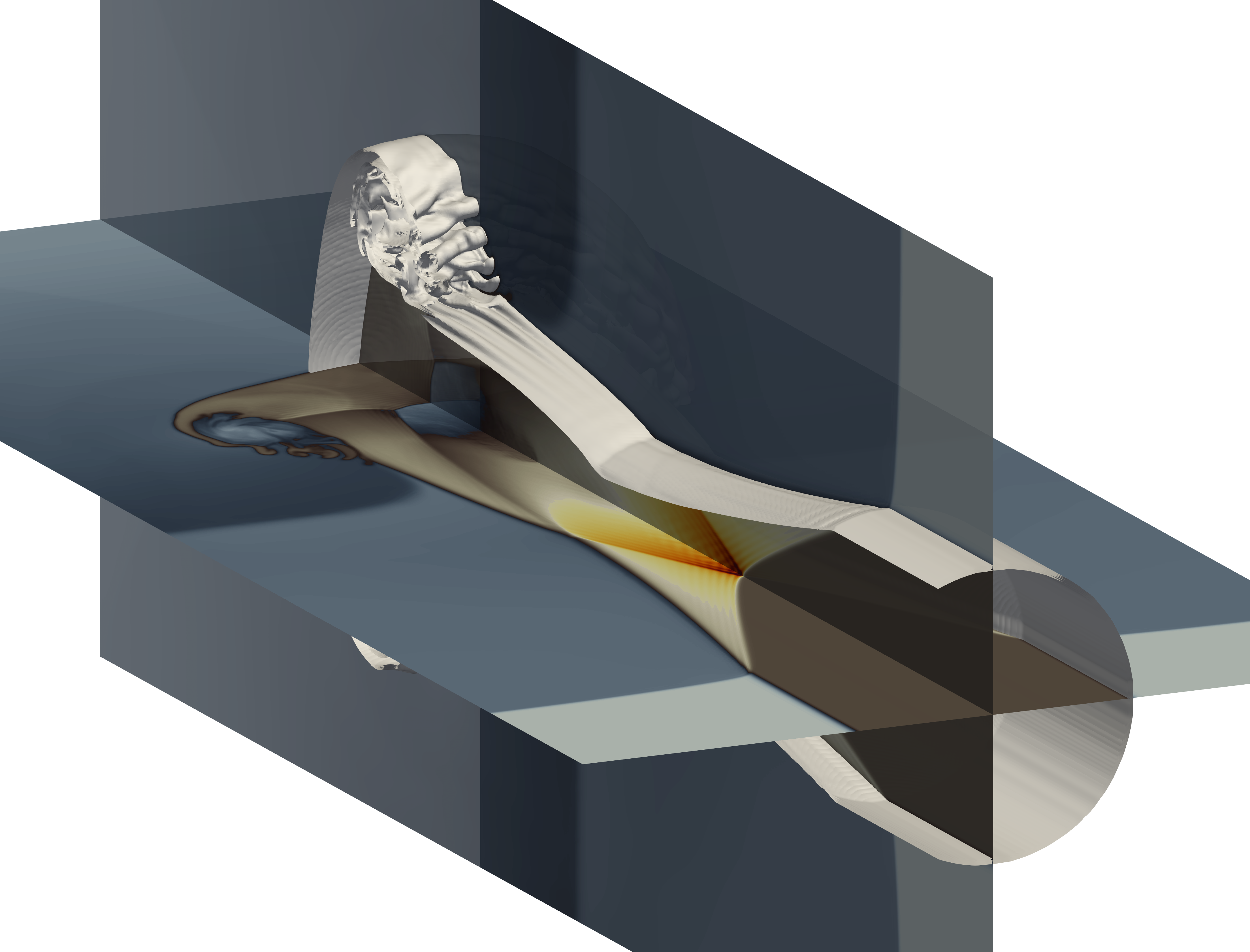}
    \caption{3D triple point problem -- Time snapshots at at $t = \{\mynum{1}{s}, \mynum{3}{s}\}$ with coarse mesh on the left and fine mesh on the right.}
    \label{fig:3d_triple_point_final}
\end{figure}

%% file: FIGS/rp1_d.tex
\begin{tikzpicture}[spy using outlines={rectangle, height=3cm,width=2.3cm, magnification=3, connect spies}]
	\begin{axis}[name=plot1,
		axis line style={latex-latex},
	    axis x line=left,
        axis y line=left,
		xlabel={$x$},
    	xmin=0, xmax=1,
    	xtick={0, 0.2, 0.4, 0.6, 0.8, 1.0},
    	ylabel={$\rho$},
    	ymin=0,ymax=1.05,
    	ytick={0, 0.2, 0.4, 0.6, 0.8, 1.0},
        minor x tick num=4,
        minor y tick num=4,
        clip mode=individual,
    	ylabel style={rotate=-90},
    	legend style={at={(0.97, 0.97)},anchor=north east,font=\small},
    	legend cell align={left},
    	style={font=\normalsize}]
    	
        \addplot[color=black, style={thin}, only marks, mark=o, mark options={scale=0.6}, mark repeat = 12, mark phase = 0]
        table[x=Points_0, y=rho, col sep=comma]{./data/rp1_exact.csv};
        \addlegendentry{Exact};
                
        \addplot[color=PlotColor1, 
                style={very thick}]
        table[x=Points_0, y=rho, col sep=comma]{./data/rp1_200.csv};
        \addlegendentry{$N = 200$};
        
        \addplot[color=PlotColor3, 
                style={very thick}]
        table[x=Points_0, y=rho, col sep=comma]{./data/rp1_400.csv};
        \addlegendentry{$N = 400$};
        
        \addplot[color=PlotColor5, 
                style={thick}]
        table[x=Points_0, y=rho, col sep=comma]{./data/rp1_800.csv};
        \addlegendentry{$N = 800$};

	\end{axis}
\end{tikzpicture}

%% file: FIGS/rp1_p.tex
\begin{tikzpicture}[spy using outlines={rectangle, height=3cm,width=2.3cm, magnification=3, connect spies}]
	\begin{axis}[name=plot1,
		axis line style={latex-latex},
	    axis x line=left,
        axis y line=left,
		xlabel={$x$},
    	xmin=0, xmax=1,
    	xtick={0, 0.2, 0.4, 0.6, 0.8, 1.0},
    	ylabel={$p$},
    	ymin=0,ymax=1.05,
    	ytick={0, 0.2, 0.4, 0.6, 0.8, 1.0},
        minor x tick num=4,
        minor y tick num=4,
        clip mode=individual,
    	ylabel style={rotate=-90},
    	legend style={at={(0.97, 0.97)},anchor=north east,font=\small},
    	legend cell align={left},
    	style={font=\normalsize}]
    	
        \addplot[color=black, style={thin}, only marks, mark=o, mark options={scale=0.6}, mark repeat = 12, mark phase = 0]
        table[x=Points_0, y=p, col sep=comma]{./data/rp1_exact.csv};
                
        \addplot[color=PlotColor1, 
                style={very thick}]
        table[x=Points_0, y=p, col sep=comma]{./data/rp1_200.csv};
        
        \addplot[color=PlotColor3, 
                style={very thick}]
        table[x=Points_0, y=p, col sep=comma]{./data/rp1_400.csv};
        
        \addplot[color=PlotColor5, 
                style={thick}]
        table[x=Points_0, y=p, col sep=comma]{./data/rp1_800.csv};

	\end{axis}
\end{tikzpicture}

%% file: FIGS/rp2_d.tex
\begin{tikzpicture}[spy using outlines={rectangle, height=3cm,width=2.3cm, magnification=3, connect spies}]
	\begin{axis}[name=plot1,
		axis line style={latex-latex},
	    axis x line=left,
        axis y line=left,
		xlabel={$x$},
    	xmin=0, xmax=1,
    	xtick={0, 0.2, 0.4, 0.6, 0.8, 1.0},
    	ylabel={$\rho$},
    	ymin=0.95,ymax=3.05,
        minor x tick num=4,
        minor y tick num=4,
        clip mode=individual,
    	ylabel style={rotate=-90},
    	legend style={at={(0.03, 0.97)},anchor=north west,font=\small},
    	legend cell align={left},
    	style={font=\normalsize}]
    	
        \addplot[color=black, style={thin}, only marks, mark=o, mark options={scale=0.6}, mark repeat = 12, mark phase = 0]
        table[x=Points_0, y=rho, col sep=comma]{./data/rp2_exact.csv};
        \addlegendentry{Exact};
                
        \addplot[color=PlotColor1, 
                style={very thick}]
        table[x=Points_0, y=rho, col sep=comma]{./data/rp2_200.csv};
        \addlegendentry{$N = 200$};
        
        \addplot[color=PlotColor3, 
                style={very thick}]
        table[x=Points_0, y=rho, col sep=comma]{./data/rp2_400.csv};
        \addlegendentry{$N = 400$};
        
        \addplot[color=PlotColor5, 
                style={thick}]
        table[x=Points_0, y=rho, col sep=comma]{./data/rp2_800.csv};
        \addlegendentry{$N = 800$};

	\end{axis}
\end{tikzpicture}

%% file: FIGS/rp2_p.tex
\begin{tikzpicture}[spy using outlines={rectangle, height=3cm,width=2.3cm, magnification=3, connect spies}]
	\begin{axis}[name=plot1,
		axis line style={latex-latex},
	    axis x line=left,
        axis y line=left,
		xlabel={$x$},
    	xmin=0, xmax=1,
    	xtick={0, 0.2, 0.4, 0.6, 0.8, 1.0},
    	ylabel={$p$},
    	ymin=0,ymax=1.05e6,
    	ytick={0, 0.2e6, 0.4e6, 0.6e6, 0.8e6, 1.0e6},
        minor x tick num=4,
        minor y tick num=4,
        clip mode=individual,
    	ylabel style={rotate=-90},
    	legend style={at={(0.97, 0.97)},anchor=north east,font=\small},
    	legend cell align={left},
    	style={font=\normalsize}]
    	
        \addplot[color=black, style={thin}, only marks, mark=o, mark options={scale=0.6}, mark repeat = 12, mark phase = 0]
        table[x=Points_0, y=p, col sep=comma]{./data/rp2_exact.csv};
                
        \addplot[color=PlotColor1, 
                style={very thick}]
        table[x=Points_0, y=p, col sep=comma]{./data/rp2_200.csv};
        
        \addplot[color=PlotColor3, 
                style={very thick}]
        table[x=Points_0, y=p, col sep=comma]{./data/rp2_400.csv};
        
        \addplot[color=PlotColor5, 
                style={thick}]
        table[x=Points_0, y=p, col sep=comma]{./data/rp2_800.csv};

	\end{axis}
\end{tikzpicture}

%% file: conclusion.tex
\section{Conclusion}
This work presents a second-order accurate, invariant-domain preserving numerical scheme for the compressible multi-species Euler equations which ensures positivity of the species densities and internal energy/pressure as well as a local minimum principle on the mixture entropy. We give the solution to the one-dimensional Riemann problem for the multi-species formulation and derive an upper bound on the maximum wave speed of the problem, which we then use to construct a robust, first-order invariant-domain preserving approximation. This approach was then extended to second-order accuracy using a modified convex limiting technique. The numerical results demonstrate the scheme’s ability to handle challenging multi-species flow problems with strong shocks and discontinuities, highlighting its potential for use in high-fidelity simulations of compressible, multi-species flow phenomena. Future work may extend this framework to include viscous effects, more complex equations of state, or reactive flow physics.

%% file: app_PTE.tex
\section{Thermal-Mechanical Equilibrium}\label{app:PTE}
For the sake of completeness, we show how the assumption of thermal equilibrium, along with Dalton's law, yields the bulk pressure of the system~\eqref{eq:pressure} independent of the pressure equilibrium assumption.
\begin{proposition} \label{prop:daltons_pressure}
    Assume thermal equilibrium holds, $T = T_k = e_k / c_{v,k}$ for all $k \in \intset{1}{n_s}$, and that the pressure is given by Dalton's law $p = \spsum \alpha_k p_k(\rho_k, e_k)$ with $p_k = (\gamma_k - 1) \rho_k e_k$.
    Then 
    \begin{equation}
        p(\rho, e, \bY) = (\gamma(\bY) - 1)\rho e,
    \end{equation}
    where $\gamma(\bY) = c_p(\bY) / c_v(\bY)$, $\rho = \spsum \alpha_k\rho_k$, and $e = \spsum Y_k e_k$.
\end{proposition}

\begin{proof}
    Using the definition of the pressure, we have the following,
    \begin{equation*}
        p = \spsum\alpha_k p_k =  \spsum \alpha_k (\gamma_k - 1) \rho_k e_k = T \spsum (\gamma_k - 1) \alpha_k \rho_k c_{v,k}.
    \end{equation*}
    Since $c_{v,k} T = e_k$ for every $k \in \intset{1}{n_s}$, then the identity holds for any convex combination of the $n_s$ terms; in particular, the identity holds for mass averaging.
    Hence, $c_v(\bY) T = e$.
    Then using that $\alpha_k\rho_k = \rho Y_k$, we have,
    \begin{equation*}
        p(\rho, e, \bY) = \frac{e}{c_v(\bY)} \spsum\alpha_k \rho_k (c_{p,k} - c_{v,k}) 
        = \rho e \frac{c_p(\bY) - c_v(\bY)}{c_v(\bY)}
        = (\gamma(\bY) - 1)\rho e,
    \end{equation*}
    which completes the proof.
\end{proof}